\documentclass{amsart}

\usepackage{hyperref}

\usepackage[margin=1in]{geometry}

\usepackage{comment}
\usepackage[alphabetic,lite]{amsrefs}
\usepackage{mathrsfs}
\usepackage{stmaryrd}

\usepackage{amssymb,amsmath,latexsym,amsthm}
\newtheorem{theorem}{Theorem}[section]
\newtheorem{lemma}[theorem]{Lemma}
\newtheorem{proposition}[theorem]{Proposition}
\newtheorem{corollary}[theorem]{Corollary}

\newtheorem{predefinition}[theorem]{Definition}
\newenvironment{definition}{\begin{predefinition}\rm}{\end{predefinition}}
\newtheorem{preremark}[theorem]{Remark}
\newenvironment{remark}{\begin{preremark}\rm}{\end{preremark}}
\numberwithin{equation}{section}

\newcommand{\Z}{\mathbb Z}
\newcommand{\N}{\mathbb N}
\newcommand \bR {{\mathbb R}}
\newcommand \bC {{\mathbb C}}
\newcommand \Q {{\mathbb Q}}
\newcommand \F {{\mathbb F}}

\DeclareMathOperator{\frob}{Frob}
\DeclareMathOperator{\Gal}{Gal}
\DeclareMathOperator{\CSp}{CSp}
\DeclareMathOperator{\Sp}{Sp}
\DeclareMathOperator{\GL}{GL}

\newcommand \LI {\mathrm{LI}}

\newcommand*\diff{\mathop{}\!\mathrm{d}}

\usepackage[usenames,dvipsnames]{color} %color comments

\title{Exceptional biases in counting primes over function fields}
\author{
    Alexandre Bailleul}%,  Lucile Devin, Daniel Keliher,  Wanlin Li}
    \address{ENS Paris-Saclay, Centre Borelli, UMR 9010, 91190 Gif-sur-Yvette, France}
     \email{alexandre.bailleul@ens-paris-saclay.fr}

\author{Lucile Devin}
\address{Univ. Littoral C\^ote d'Opale, UR 2597 LMPA, Laboratoire de Math\'ematiques Pures et Appliqu\'ees Joseph Liouville, F-62100 Calais, France}
    \email{lucile.devin@univ-littoral.fr}

\author{Daniel Keliher}
\address{University of Georgia, Department of Mathematics,  200 D. W. Brooks Drive, Athens, GA 30602, USA}
    \email{keliher@uga.edu}

\author{Wanlin Li}
 \address{Washington University in St. Louis, Department of Mathematics and Statistics, One Brookings Drive, 
St. Louis, MO 63130, USA}
    \email{wanlin@wustl.edu}

\date{}

\begin{document}

\maketitle

\begin{abstract}
We study how often exceptional configurations of irreducible polynomials over finite fields occur in the context of prime number races and Chebyshev's bias. In particular, we show that three types of biases, which we call ``complete bias'', ``lower order bias'' and ``reversed bias'', occur with probability going to zero among the family of all squarefree monic polynomials of a given degree in $\F_q[x]$ as $q$, a power of a fixed prime, goes to infinity. The bounds given improve on a previous result of Kowalski, who studied a similar question along particular 1-parameter families of reducible polynomials. The tools used are the large sieve for Frobenius developed by Kowalski, an improvement of it due to Perret-Gentil and considerations from the theory of linear recurrence sequences and arithmetic geometry.
\end{abstract}

\section{Introduction}

Chebyshev's bias is the phenomenon that there are more prime numbers of the form $4n+3$ than of the form $4n+1$ in initial intervals $\llbracket 2, x \rrbracket$ of $\mathbb{N}$ for most values of $x$ (more precisely, the set of such $x$ admits a logarithmic density of around $99.59\%$). More generally, primes which are congruent to a fixed non-square residue class modulo an integer $q$ are more numerous than those which are congruent to a given square residue class modulo $q$ in initial intervals of $\mathbb{N}$. The origin of this phenomenon was explained by Rubinstein and Sarnak in \cite{RS}.

The analogue of Chebyshev's bias over function fields was first considered by Cha in \cite{Cha2008} to study inequities in the distribution of irreducible polynomials in residue classes of $\F_q[x]$, and later by Cha and Im in \cite{ChaIm2011} in function field extensions. As in the classical archimedean case of \cite{RS}, a central hypothesis is a linear independence hypothesis which will be called $\LI$ throughout. If the arguments of the non-trivial inverse zeros (of non-negative imaginary parts) of the underlying $L$-functions are of the form $\sqrt{q}e^{i\theta}$, then $\LI$ claims that the $\theta$'s, together with $\pi$, are linearly independent over $\Q$. A consequence of $\LI$ is that Chebyshev's bias favours non-square residue classes rather than square residue classes in the distribution of primes. See \cite{RS}*{page~185} (where it is called $\mathrm{GSH}$) for the archimedean case, and \cite{Cha2008}*{page~1366} for the function field case.  For a survey on prime number races over $\Q$, see \cite{GranvilleMartin}.

Over $\Q$ and number fields, exceptional biases have been studied in the literature, notably in a series of papers by Ford and Konyagin \cites{Ford_Konyagin_2002,Ford_Konyagin_2003} and \cite{FordKonyaginLamzouri}.
Fiorilli and Martin \cite{FiorilliMartin}, under both the Generalized Riemann Hypothesis and $\LI$, list the largest possible biases in the prime number race between quadratic residues and non-quadratic residues. In number field extensions, Bailleul produced infinite families of examples exhibiting a reversed bias in \cite{Bailleul1}, conditionally on a suitable linear independence hypothesis. As for unconditional results, Fiorilli and Jouve constructed infinite families exhibiting a complete bias in \cite{FiorilliJouve}.

The state of affairs in the function field setting is rather different. For instance, over $\mathbb{F}_q[x]$, there are a few known counterexamples to $\LI$ (see \cite{Cha2008}*{Section 5}, \cite{DevinMeng}*{Section 3}, \cite{Dupuy et al}*{Section 7}, \cite{Sedrati}*{Section 10}), which can lead to what we call ``exceptional biases'', for example favouring square residue classes rather than non-square residue classes (``reversed bias''), or having more non-square residue classes than square residue classes 100\% of the time (``complete bias''). In \cite{CFJ2016}, Cha, Fiorilli and Jouve give examples of exceptional biases in Mazur's race related to counting points on elliptic curves.
They prove also the genericity of $\LI$ for certain families in this context in \cite{CFJ2017}.

In this paper, we investigate three types of exceptional biases. For those types of biases, we establish more precise necessary conditions than negation of $\LI$ for them to hold, and we show that they happen very rarely. 

In order to state our results more precisely, we need to introduce some notation. When $q$ is a power of a prime $p$ and $n \geq 1$, we let $$\mathcal{H}_n(\F_q)=\{ f\in \F_q[x] \mid f \text{ is monic, squarefree}, \deg f=n \}.$$ For $f \in \mathcal{H}_n(\F_q)$,  let $\chi_f$ denote the unique primitive quadratic character modulo $f$, and \begin{align*}\Pi(n; \chi_f)  &:=\ \frac{n}{q^{n/2}} \Big(\#\{h \in \F_q[x] \mid h \text{ is irreducible, } \deg h = n \text{ and } \chi_f(h)=1\}\\ &- \#\{h \in \F_q[x] \mid h \text{ is irreducible, } \deg h = n \text{ and } \chi_f(h)=-1\}\Big).\end{align*} Note that when $f$ is irreducible then this is, up to a positive factor, the difference between the number of irreducible square residues modulo $f$ of degree $n$ and those which are non-square residues. We also denote by $\mathcal{C}_f$ the hyperelliptic curve defined over $\F_q$ as the smooth projective model of the curve with affine equation $y^2 = f(x)$.  

In \cite{Kowalski2010}, Kowalski showed that, in a precise quantitative sense (see formula \eqref{Kowalski} below), the $\LI$ hypothesis is generically true for the zeta functions of hyperelliptic curves of the form $\mathcal{C}_{g(x)(x-t)}$, where $g \in \mathcal{H}_n(\F_q)$ of even degree is fixed and $t \in \F_q$ is a parameter such that $g(t) \neq 0$, as $q \to \infty$.
This implies that for most of the parameters $t$, the counting function $\Pi(n; \chi_{g(x)(x-t)})$ is biased towards negative values and changes sign infinitely many times. This behavior is expected to hold for $\Pi(n; \chi_f)$ generically among $f \in \mathcal{H}_n(\F_q)$ because of $\LI$.

Our main results are the following four bounds, which improve Kowalski's result. The terms ``complete bias'', ``lower order bias'', and ``reversed bias'' are defined, respectively, in Definitions \ref{complete}, \ref{lower}, and \ref{wrongdirection}.

\begin{theorem}\label{thm:main}
Let $p$ be an odd prime number, $q$ a power of $p$ and $n \geq 1$. We write $g = \left \lfloor \frac{n-1}{2} \right \rfloor$ and 
 $A = 2g^2+g+2$.
\begin{enumerate}
    \item\label{item LI main} We have $$\frac{1}{|\mathcal{H}_n(\F_q)|} \#\{f \in \mathcal{H}_n(\F_q) \mid \text{ The zeta function of } \mathcal{C}_f \text{ does not satisfy } \LI\} \ll_{p, g} q^{-\frac{1}{2A}} (\log q)^{1-\delta}$$ where $1 \geq \delta \underset{g \to +\infty}{\sim} \frac{1}{8g}$.
    \item\label{item complete bias th main} If $q$ is a square then, we have $$\frac{1}{|\mathcal{H}_n(\F_q)|} \#\{f \in \mathcal{H}_n(\F_q) \mid \Pi(n; \chi_f) \text{ exhibits a complete bias}\} \ll_{p, g} q^{-\frac{1}{A}} \log q,$$
    and $\#\{f \in \mathcal{H}_n(\F_q) \mid \Pi(n; \chi_f) \text{ exhibits a complete bias}\} = 0$ otherwise.
    \item\label{item lower order th main}  We have $$\frac{1}{|\mathcal{H}_n(\F_q)|} \#\{f \in \mathcal{H}_n(\F_q) \mid \Pi(n; \chi_f) \text{ exhibits a lower order bias}\} \ll_{p, g} q^{-\frac{1}{A}} \log q.$$ 
    \item\label{item wrong direction th main}  We have $$\frac{1}{|\mathcal{H}_n(\F_q)|} \#\{f \in \mathcal{H}_n(\F_q) \mid \Pi(n; \chi_f) \text{ exhibits a reversed bias}\} \ll_{p, g} q^{-\frac{1}{2A}} (\log q)^{1-\delta'},$$ where $1 \geq \delta' \underset{g \to +\infty}{\sim} \frac{7}{24g}$. 
\end{enumerate}
\end{theorem}

To prove this theorem, we follow closely Kowalski's method based on the large sieve for Frobenius developed in \cite{Kowalski book} (and improved by Perret-Gentil \cite{Perret-GentilANT}). The theorem above should be compared to Kowalski's bound \eqref{Kowalski}, which we now state.

\begin{theorem}[\cite{Kowalski2010}*{Proposition~1.1}]\label{thm:Kowalski}
Let $g \geq 1$ be an integer, and let $f \in \Z[x]$ be a squarefree monic polynomial of degree $2g$. Let $p$ be an odd prime such that $p$ does not divide the discriminant of $f$, and let $U/\F_p$ be the open subset of the affine $t$-line where $f(t) \neq 0$. Consider the algebraic family $\mathcal{C}_f \to U$ of smooth projective hyperelliptic curves of genus $g$ given as the smooth projective models of the curves with affine equations $$C_t : y^2 = f(x)(x-t), \quad \text{ for } t \in U.$$ Then for any extension $\F_q/\F_p$ we have \begin{equation} \label{Kowalski} \frac{1}{|U(\mathbb{F}_q)|} \#\{t \in U(\mathbb{F}_q) \mid \text{ The zeta function of } C_t \text{ does not satisfy } \LI\} \ll_g q^{-\frac1{2A}} (\log q)^{1-\delta},\end{equation} where $A = 2g^2 + g + 2$ and $1 \geq \delta \underset{g \to +\infty}{\sim} \frac{1}{8g}$.
\end{theorem}

\begin{remark}
The bound stated in \cite{Kowalski2010} is a bit larger, the exponent of $\log q$ is simply $1$, but Kowalski gave this better exponent in \cite{Kowalski book}*{Theorem~8.15}, for the more general condition that the Galois group of the zeta function of $\mathcal{C}_t$ is not maximal. It is indeed more general since if there exists a non-trivial linear relation between $\pi$ and the arguments of the roots of the zeta function, hence a multiplicative relation between those roots, then its Galois group is not maximal since this relation cannot be preserved by every allowed permutations of the roots. However, note there is a typo in the bound stated in \cite{Kowalski book}*{Theorem~8.15}: the exponent there reads $1-\delta$ with $\delta \underset{g \to +\infty}{\sim} \frac{1}{4g}$, coming from the larger contribution of $\delta_2 \geq \frac{1}{4g}$ p.181, but we can actually only get $\delta_2 \geq \frac{1}{8g}$. The count is detailed in \cite{Kowalski2006}*{Lemma 7.3 iii)} but the author is counting each symplectic polynomial with a given factorization twice, hence a missing $\tfrac{1}{2}$ factor. The proof of Lemma \ref{lemma counting polynomials} fixes this.
\end{remark}

The bounds in Theorem~\ref{thm:main} improve Kowalski's bound \eqref{Kowalski} in two aspects. First, the space of parameters is larger than Kowalski's. While he obtains his bound for families of polynomials of a very specific shape, our bound applies to all monic squarefree polynomials of a given degree. It should be noted that our method would allow us to prove the same bounds as in Theorem \ref{thm:main} but along Kowalski's family of curves in Theorem \ref{thm:Kowalski}, independently of $p$, by using the large sieve estimate \cite{Kowalski book}*{Corollary 8.10} instead of Proposition~\ref{Prop Frobenius large sieve} of this paper. Moreover, the exponents for $q$ in the bounds \ref{thm:main}.\ref{item complete bias th main} and \ref{thm:main}.\ref{item lower order th main}  are twice as small, while the exponent for $\log q$ in the last bound \ref{thm:main}.\ref{item wrong direction th main} is slightly better.
Observe however that by passing to a multidimensional space of parameters, we lose the uniformity in $p$ in the bounds. Such a phenomenon was already present in \cite{Kowalski book}*{Corollary~8.10} which results in a larger exponent of $q$ in the multidimensional case. In our case, the uniformity in $p$ is lost when applying the improved bound \cite{Perret-GentilANT}*{Theorem 5.14.(ii).(c)}. 

For the first two properties considered in Theorem \ref{thm:main}, inputs from arithmetic geometry give us better bounds for some restricted genera.
Our first improvement is for genus $1$ or $2$ concerning the failure of $\LI$.

\begin{theorem}\label{main:small genus}
Let $p \neq 2,3$ be a prime number, $q$ a power of $p$ and $3 \leq n \leq 6$. We write $g = \left \lfloor \frac{n-1}{2} \right \rfloor$, so that $1 \leq g \leq 2$.
When $g=1$, we have $$\frac{1}{|\mathcal{H}_n(\F_q)|} \#\{f \in \mathcal{H}_n(\F_q) \mid \text{ The zeta function of } \mathcal{C}_f \text{ does not satisfy } \LI\} \ll \frac{p}{q}.$$ 
When $g=2$, then we have $$\frac{1}{|\mathcal{H}_n(\F_q)|} \#\{f \in \mathcal{H}_n(\F_q) \mid \text{ The zeta function of } \mathcal{C}_f \text{ does not satisfy } \LI\} \ll_p q^{-\frac{1}{12}} \log q.$$
\end{theorem}

In particular, in this more restricted setting,  these bounds improve on \ref{thm:main} \ref{item LI main} and {\it a fortiori} on \ref{thm:main} \ref{item wrong direction th main}. Note that the result for genus at most two comes from the fact that we completely understand the Frobenius eigenvalues for genus $1$ and $2$ hyperelliptic curves over $\overline{\mathbb{F}}_p$. The reason is that all smooth projective curves of genus at most two are hyperelliptic, and the Torelli image of $\mathcal{M}_2$ is dense in $\mathcal{A}_2$. Neither of the facts holds for higher genus.

Our last result is a bound for the bias dealt with in Theorem \ref{thm:main} \ref{item complete bias th main} which is uniform in the degree, at the expense of being worse in terms of $q$ for small $g$. 

\begin{theorem}\label{thm:complete arithm geom}
If $q=p^e$ is a fixed prime power with $2 \mid e$. Then,
$$ \sup_{n \geq 3} 
\frac{1}{|\mathcal{H}_n(\F_q)|} \#\{f \in \mathcal{H}_n(\F_q) \mid \Pi(n; \chi_f) \text{ exhibits a complete bias}\} \ll \frac{1}{q^{1/276}}.$$
\end{theorem}

In particular, this bound is better than the second bound of Theorem \ref{thm:main} in terms of $q$ as soon as $g \geq 12$. The underlying method coming from arithmetic geometry cannot deal with the conditions in \ref{thm:main} \ref{item lower order th main} and \ref{item wrong direction th main} because they are concerned with multiple zeros of the zeta function of $\mathcal{C}_f$ at once.

\subsection*{Outline of the paper}

In Section~\ref{section : prelim} we set the notation and give preliminary results used in the rest of the paper. 
In particular, section~\ref{recseq} states some results about linear recurrent sequences, and section~\ref{subsec : large sieve} is devoted to the proof of a large sieve statement, which is one important step in the proof of Theorem~\ref{thm:main}.
In Section~\ref{sec : LI} we give a proof of the first item of Theorem~\ref{thm:main} following Kowalski's method and Theorem~\ref{main:small genus} by elementary methods. In Section~\ref{sec :complete} we derive conditions for a complete bias and prove the second item of Theorem~\ref{thm:main} with the large sieve for Frobenius and Theorem~\ref{thm:complete arithm geom} with arithmetic geometry. In Section~\ref{sec : lower order} and~\ref{sec : wrong direction} we derive conditions for a lower order bias and a reversed bias respectively and we prove the last two items of Theorem~\ref{thm:main}. Finally, in Section~\ref{sec : counting} we gather counting lemmas obtained using our large sieve result Proposition~\ref{Prop Frobenius large sieve} that are used in the proofs of the different parts of Theorem~\ref{thm:main}.

\subsection*{Acknowledgements}
This work was partially supported by the grant KAW 2019.0517 from the Knut and Alice Wallenberg Foundation (for LD). Part of this work was conducted while WL was in residence at the Mathematical Sciences Research Institute in Berkeley, California, during the Spring 2023 semester.
The authors thank Florent Jouve, Jordan Ellenberg and Emmanuel Kowalski for helpful discussions. They also thank Régis de la Bretèche for organizing the elementary and analytic number theory seminar in IHP, Paris, where ideas used in this paper were born.

\section{Preliminary results and notations}\label{section : prelim}

\subsection{Notations and Definitions}

We first provide notations for the rest of the paper. When $f \in \mathcal{H}_n(\F_q)$, the projective curve with affine model $y^2 = f(x)$ is denoted by $\mathcal C_f$. Recall that $\mathcal C_f$ has genus $g = \left \lfloor \frac{n-1}{2} \right \rfloor$.

For $f\in \mathcal{H}_n(\F_q)$, let $\chi_f$ be the primitive quadratic character modulo $f$.  We want to compare the number of degree $n$  irreducible polynomials $P$ over $\F_q$ such that $\chi_f(P)=1$ and those such that $\chi_f(P)=-1$ for varying $n$. Define the Dirichlet $L$-function associated to a Dirichlet character $\chi$ modulo $f$ as
$$L(s, \chi) = \sum_{a \text{ monic }} \frac{\chi(a)}{\lvert a \rvert^{s}}= \prod_{P \text{ irreducible }}\left( 1- \frac{\chi(P)}{\lvert P \rvert^s}\right)^{-1},$$
where $\lvert a \rvert = q^{\deg a}$ and the sum and product above range over monic (resp. irreducible) polynomials of $\F_q[x]$. 

We now recall some properties of the $L$-functions under consideration; see e.g. \cite{Rosen2002}*{Proposition~4.3, Theorem~5.9} for details. For a non-principal Dirichlet character $\chi$, the Dirichlet~$L$-function  $L(s, \chi)$ is a polynomial $\mathcal L(u, \chi)$ in $u := q^{-s}$ with integer coefficients and the zeta function of $\mathcal{C}_f$ is a rational function in $u$, which we denote by $\zeta(\mathcal{C}_f, u) = \frac{Z_{f}(u)}{(1-u)(1-qu)}$.  Thanks to the deep work of Weil \cite{Weil_RH}, we know the analogue of the Riemann Hypothesis is satisfied for these zeta functions, that is their inverse zeros have absolute value $\sqrt q$. When $n$ is odd, then $\mathcal{L}(u, \chi_f) = Z_f(u)$, and when $n$ is even, we have $\mathcal{L}(u, \chi_f) = Z_f(u)(1-u)$. In the following, we will mostly use the reciprocal polynomial 
\begin{equation}\label{eq def P_f}
  P_f(T) = T^{2g} Z_f(T^{-1}),  
\end{equation}
 which is monic, and its roots are the inverse zeros of $Z_f$.

In the following, we denote by $\alpha_{j}(\chi)= \sqrt{q}e^{i\theta_{j}(\chi)}$ the distinct inverse zeros of $\mathcal{L}(u,\chi)$ of norm $\sqrt{q}$, with multiplicity $m_{\theta_j}(\chi)$.
We might forget the dependency in the character $\chi$ when only one character is considered and the notation stays clear from the context.
We let $r$ be the number of distinct pairs of conjugate non-real zeros of $\mathcal{L}(u, \chi_f)$. Since $\mathcal{L}(u, \chi_f)$ has real coefficients, after reordering, we can assume $\theta_{j+r}(\chi_f) = - \theta_j(\chi_f)$ and we have $m_{\theta_j}(\chi_f) = m_{-\theta_j}(\chi_f)$ for $1 \leq j \leq r$. Since $\chi_f$ is primitive, we have $$m_0(\chi_f) + m_{\pi}(\chi_f) + 2 \sum_{j=1}^r m_{\theta_j}(\chi_f) = 2g.$$

Using the explicit formula in \cite{Cha2008}*{Proposition~4.2}, our object of study is the function
    \begin{align}\label{eq:Pichi}
    \Pi(n; \chi_f)  &:=\ \frac{n}{q^{n/2}} \Big(\#\{h \in \F_q[x] \mid h \text{ is irreducible , } \deg h = n \text{ and } \chi_f(h)=1\} \nonumber \\
    &\quad - \#\{h \in \F_q[x] \mid h \text{ is irreducible , } \deg h = n \text{ and } \chi_f(h)=-1\}\Big) \nonumber \\
    &=\ \frac{n}{q^{n/2}}\sum_{\substack{\deg h = n \\ h \text{ irreducible}}} \chi_f(h) \nonumber \\
    &=\ -  \left(m_{0}(\chi_f) +\tfrac{1}{2}\right)  -\left(m_{\pi}(\chi_f)+\tfrac{1}{2}\right) (-1)^{n} -\sum_{\theta_j\neq 0, \pi} m_{\theta_j}(\chi_f)  e^{in\theta_{j}(\chi_f)}   + O_{f}\left(q^{-\frac{n}{6}}\right).
	\end{align}
	
Let $\Delta_f(n)$ be the opposite of the main sum of $\Pi(n; \chi_f)$ in \eqref{eq:Pichi}; that is \begin{equation}
    \label{eq:Delta}
\Delta_f(n) =   \left(m_{0}(\chi_f) +\tfrac{1}{2}\right)  +\left(m_{\pi}(\chi_f)+\tfrac{1}{2}\right) (-1)^{n} +\sum_{\theta_j\neq 0, \pi} m_{\theta_j}(\chi_f)  e^{in\theta_{j}(\chi_f)}.
\end{equation}
In the case the set $\lbrace \theta_1(\chi_f), \dots, \theta_r(\chi_f) \rbrace\cup \lbrace \pi\rbrace$ is linearly independent over $\mathbb{Q}$, which is expected to be the generic case, then $\Delta_f(n) -  \big(m_{0}(\chi_f) +\tfrac{1}{2}\big) $  oscillates around zero and takes positive (resp. negative) values half of the time (i.e., for $50\%$ of positive integers~$n$). 
Thus, $\Delta_f$ is larger (resp. smaller) than its mean value $m_0(\chi_f) + \tfrac{1}{2}$ for half of the positive integers $n$. 
One  deduces (see \cite{Cha2008}*{page~1366}) that there is a bias in the distribution of the values of $\Delta_f$ in the direction of positive values, $\textit{i.e.}$ coming back to $\Pi(n; \chi_f)$ we expect a bias towards negative values. Or in other terms, there are in general more irreducible polynomials~$P$ of degree~$n$ with $\chi_f(P) = -1$ than with $\chi_f(P) =1$.

Now, it can happen that the oscillating part does not distribute so well between positive and negative values. This is the case in the examples given in \cite{Cha2008}*{Section~5} and also for the different kinds of behaviors we consider in this paper.

\begin{remark}
In this paper, we are studying the summatory function of a quadratic character over irreducible polynomials. Another ``prime number race'' of interest is the one between quadratic residues and non-quadratic residues. 
Observe that these are the same in the case $f$ is irreducible. In the case $f$ is not irreducible, one has to take into account the contribution of all quadratic (non-necessarily primitive) characters modulo $ f$, which makes the study more difficult.  The general formula proved in \cite{DevinMeng}*{Proposition 5.2} is 
\begin{align*} \Pi(n; f, \square,\boxtimes)  :=&\   \frac{n}{q^{n/2}} \Big(\frac{1}{\lvert \square\rvert}\lvert \lbrace h \in \F_{q}[x] \mid  h \text{ monic irreducible, } \deg{h} = n, h \bmod f \in \square \rbrace\rvert \\
   &- \frac{1}{\lvert \boxtimes\rvert}\lvert \lbrace h \in \F_{q}[x] \mid  h \text{ monic irreducible, } \deg{h} = n, h \bmod f \in \boxtimes \rbrace\rvert \Big) \nonumber\\
    =&\  \frac{-1}{\lvert \boxtimes\rvert} \Bigg\{ \sum_{\chi \in X_f^{\text{quad}}}\Bigg(\ \left(m_{0}(\chi) +\tfrac{1}{2}\right)  +\left(m_{\pi}(\chi)+\tfrac{1}{2}\right) (-1)^{n} +\sum_{\theta_j\neq 0, \pi} m_{\theta_j}(\chi)  e^{in\theta_{j}(\chi)} \Bigg) \\
    &\hspace{2cm} + O_{f}\left(q^{-\frac{n}{6}}\right) \Bigg\},\nonumber
	\end{align*}
where $\square$ denotes the set of quadratic residues modulo $f$, $\boxtimes$ denotes the set of non-quadratic residues modulo $f$ and $X_f^{\text{quad}}$ is the set of quadratic characters modulo $f$.
\end{remark}

For a given $f\in \mathcal{H}_n(\F_q)$, we define three kinds of ``exceptional biases'' as follows.

\begin{definition}[Complete bias] \label{complete}
We say that $\Pi(n, \chi_f)$ exhibits a \emph{complete bias} if $\Delta_f(n)>0$ for almost all $n$. That is,
$$\mathrm{dens}(\Delta_f>0) := \lim_{X \to +\infty} \frac{1}{X} \sum_{n \leq X} \mathbf{1}_{\Delta_f(n) > 0} = 1.$$ \end{definition}

\begin{remark}[$\Pi(n, \chi_f)$ vs. $\Delta(n, \chi_f)$]
In particular, if $\Pi(n, \chi_f)$ exhibits a complete bias, then $\mathrm{dens}(\Pi(n, \chi_f) < 0)$ exists and is equal to $1$, but the converse need not hold. Note that the above definition may not cover all the cases for which $\overline{\mathrm{dens}}(\Pi(n, \chi_f)<0 )=1 $ : it may happen that $\Delta_f(n)=0$ for a positive proportion of $n$ and then for those $n$, the sign of $\Pi(n, \chi_f)$ is determined by the sign of the error term $O_f\left(q^{- \frac{n}{6}}\right)$ and necessitate further study. In the next definition, we define the case of ``lower order bias'' below to characterize this possibility.
\end{remark}

\begin{definition}[Lower order bias] \label{lower}
We say that $\Pi(n, \chi_f)$ exhibits a \emph{lower order bias} if $\Delta_f(n)=0$ for a positive proportion of $n$. That is, $$\mathrm{dens}(\Delta_f(n)=0) := \lim_{X \to +\infty} \frac{1}{X} \sum_{n \leq X} \mathbf{1}_{\Delta_f(n) = 0} > 0.$$
\end{definition}

\begin{remark}
The condition of having a lower order bias is close to the condition ``ties have positive density'', as introduced by Martin and Ng in \cite{MartinNg2020} in the context of prime number races.
\end{remark}

Finally, the last type of exceptional bias we are going to study is a direct incompatibility with the expectation that $\Pi(n, \chi_f)$ is negative for more than $50\%$ of  integers $n$.

\begin{definition}[Reversed bias] \label{wrongdirection}
We say that $\Pi(n, \chi_f)$ exhibits a \emph{reversed bias} if $\Delta_f(n) <0$ for more than half of the $n$. That is, $$\mathrm{dens}(\Delta_f(n) < 0) := \lim_{X \to +\infty} \frac{1}{X} \sum_{n \leq X} \mathbf{1}_{\Delta_f(n) < 0} > \frac12.$$ \end{definition}

\begin{remark}
\begin{enumerate}
    \item[]
    \item In Section \ref{recseq}, we will show the three densities in Definitions \ref{complete},\ref{lower},\ref{wrongdirection} exist, see Corollaries \ref{Cor : density lower order} and \ref{cor:complete}.
    \item Note that both a lower order bias and a reversed bias may occur simultaneously, but that is the only possible combination of two exceptional biases.
\end{enumerate}
\end{remark}

\begin{remark}\label{Rk: Cha's counting}
Observe that we could also (as in \cites{Cha2008,DevinMeng})
count irreducible polynomials of degree $\leq n$ instead of degree $=n$.
In this case, the functions replacing $\Pi(n;f,\square,\boxtimes)$ and $\Pi(n,\chi_f)$ take the following more complicated forms:
\begin{align*} \Pi(\leq n; f, \square,\boxtimes)  &:=\   \frac{n}{q^{n/2}} \Big(\frac{1}{\lvert \square\rvert}\lvert \lbrace h \in \F_{q}[x] \mid  h \text{ monic irreducible, } \deg{h} \leq n, h \bmod f \in \square \rbrace\rvert \\
    &- \frac{1}{\lvert \boxtimes\rvert}\lvert \lbrace h \in \F_{q}[x] \mid  h \text{ monic irreducible, } \deg{h} \leq n, h \bmod f \in \boxtimes \rbrace\rvert \Big) \\
    &=\  \frac{-1}{\lvert \boxtimes\rvert} \Bigg\{ \sum_{\chi \in X_f^{\text{quad}}}\Bigg(\ \left(m_{0}(\chi) +\tfrac{1}{2}\right)\frac{\sqrt{q}}{\sqrt{q}-1}  +\left(m_{\pi}(\chi)+\tfrac{1}{2}\right)\frac{\sqrt{q}}{\sqrt{q}+1}  (-1)^{n} \\ &+\sum_{\theta_j\neq 0, \pi} m_{\theta_j}(\chi) \frac{\sqrt{q}e^{i\theta_{j}(\chi)} }{\sqrt{q}e^{i\theta_{j}(\chi)} -1}  e^{in\theta_{j}(\chi)} \Bigg) + O_{f}\left(q^{-\frac{n}{6}}\right) \Bigg\};\\
\Pi(\leq n; \chi_f)  :=&\   \frac{n}{q^{n/2}}\sum_{\substack{\deg h \leq n \\ h \text{ irreducible}}} \chi_f(h) \\
    =&\ -  \left(m_{0}(\chi_f) +\tfrac{1}{2}\right)\frac{\sqrt{q}}{\sqrt{q}-1}   -\left(m_{\pi}(\chi_f)+\tfrac{1}{2}\right)\frac{\sqrt{q}}{\sqrt{q}+1}  (-1)^{n}  \\
    &- \sum_{\theta_j\neq 0, \pi} m_{\theta_j}(\chi_f) \frac{\sqrt{q}e^{i\theta_{j}(\chi_f)} }{\sqrt{q}e^{i\theta_{j}(\chi_f)} -1}  e^{in\theta_{j}(\chi_f)}+ O_{f}\left(q^{-\frac{n}{6}}\right),
	\end{align*}
where the sums are over $\lbrace h \in \F_{q}[x] \mid  h \text{ monic irreducible, } \deg{h} \leq n\rbrace$.

We cannot adapt most of our proofs for those quantities. For instance, the maximal value of such a sum is not easy to determine, and we'll make frequent use of the maximum values in Section \ref{subsec:completeconditions} (e.g. the proof of Proposition~\ref{lemma:necessary-complete} to see why maximal values are relevant to us). However, we have for example $\Delta(\leq n; f, \square,\boxtimes) = \Delta(n; f, \square,\boxtimes) + O\left(\frac{\sum_{\theta} m_{\theta}(\chi_f)}{\sqrt q}\right)$, where $\Delta(\leq n; f, \square,\boxtimes)$ represents the main sum in $\Pi(\leq n; f, \square,\boxtimes)$ above, and so if $q$ is large enough compared to $\sum_{\theta} m_{\theta}(\chi_f)$, the sign of $\Delta(\leq n; f, \square,\boxtimes)$ is the sign of $\Delta(n; f, \square,\boxtimes)$. In particular, under the right conditions, a complete bias and a reversed bias in the ``degree $=n$'' setting one gets from studying $\Pi(n; \chi_f$),  implies a similar bias in the ``degree $\leq n$'' setting one gets from studying $\Pi(\leq n; \chi_f$). Note also that the difference between counting irreducible polynomials of degree equal to $n$ and counting those of degree at most $n$ is analogous to the difference between counting prime number in intervals of the form $[X, 2X]$ and those in intervals of the form $[2, X]$.
\end{remark}

\subsection{Properties of limiting distributions}

To study the densities involved in the definitions \ref{complete}, \ref{lower}, and \ref{wrongdirection}, we will use the notion of limiting distribution, which we define as follows.
\begin{definition}
Let $D:\mathbb{N}\rightarrow\mathbb{R}$ be a real function, 
we say that $D$ admits a limiting distribution if there exists a probability measure $\mu$ on Borel sets in $\mathbb{R}$ such that
for any bounded continuous function $h$ on $\mathbb R$, we have
\begin{align*}
\lim_{Y\rightarrow\infty}\frac{1}{Y}\sum_{n\leq Y}h(D(n)) = 
\int_{\mathbb{R}}h(t)\diff\mu(t).
\end{align*}
We call $\mu$ the limiting distribution of the function $D$.
\end{definition}

The function $\Delta_f$ defined as Equation \ref{eq:Delta} is quasi-periodic, and we can apply the Kronecker-Weyl equidistribution theorem (see e.g. \cite{Hum}*{Lemma~2.7} and \cite{Bailleul2}*{Theorem~2.2}) to prove the following proposition (\cite{DevinMeng}*{Proposition 2.1}).

\begin{proposition}
\label{Prop_limitingDist}\label{muDeltaf}
The function $\Delta_f$ admits a limiting distribution $\mu_{\Delta_f}$ with mean value $m_{0}(\chi_f) + \frac{1}{2}$ and variance  $$\big(m_{\pi}(\chi_f\big)+ \tfrac12)^2 + \frac12\sum\limits_{j=1}^r m_{\theta_j}(\chi_f) ^2.$$
Moreover, the measure $\mu_{\Delta_f}$ has support in $$\left[m_0(\chi_f) - m_{\pi}(\chi_f) - 2\sum\limits_{j=1}^{r}m_{\theta_j}(\chi_f) , ~ m_0(\chi_f) + m_{\pi}(f) + 1  +2 \sum\limits_{j=1}^r m_{\theta_j}(\chi_f) \right].$$
\end{proposition}

The next lemma will be used to study reversed bias.

\begin{lemma}\label{lem : symmetry}
The distribution $\mu_{\Delta_f}$ in Proposition \ref{muDeltaf} is symmetric with respect to $m_{0}(\chi_f) + \frac{1}{2}$ if and only if there is no relation $$k_0\pi + \sum_{j=1}^r k_j\theta_j \equiv 0 \mod 2\pi$$ with $k_0, \dots, k_r \in \Z$ and $k_0 + \sum_{j=1}^r k_j \equiv 1 \mod2$.
\end{lemma}

\begin{proof} Denote by $A(\Delta_f)$ the closure of the $1$-parameter group $H := \lbrace n(\pi,\theta_{1},\ldots,\theta_{r}) : n\in\mathbb{Z}\rbrace/(2\pi\mathbb{Z})^{r+1}$ in the $(r+1)$-dimensional torus $\mathbb{T}^{r+1}:= (\mathbb{R}/2\pi\mathbb{Z})^{r+1}$. We first remark that by Pontryagin duality, for any $\underline z \in \mathbb{T}^{r+1}$, $\underline z \in A(\Delta_f)$ if and only if for every character $\underline k = (k_0, \dots, k_r) \in H^{\bot} \subset \Z^{r+1}$, one has $\underline k(\underline z) = k_0 z_0 + \dots + k_r z_r = 0$. Therefore, we just need to show that $\mu_{\Delta_f}$ is symmetric with respect to $m_0(\chi_f) + \tfrac{1}{2}$ if and only if $(\pi, \dots, \pi) \in A(\Delta_f)$, since $\underline k(\pi, \dots, \pi) = 0$ if and only if $\sum_{i=0}^r k_i$ is even.

By the Kronecker--Weyl Equidistribution Theorem (see for example \cite{DevinMeng}*{Lemma 2.2}),  $A(\Delta_f)$ is a subtorus of $\mathbb{T}^{r}$ and we have,
	for any continuous function $h: \mathbb{T}^{r}\rightarrow \mathbb{C}$,
	\begin{equation*}
\lim_{Y\rightarrow\infty}\frac{1}{Y}\sum_{n=0}^{Y}h(n\pi,n\theta_{1},\ldots,n\theta_{r})
	= \int_{A(\Delta_f)}h(a)\diff\omega_{A(\Delta_f)}(a)
	\end{equation*}
	where $\omega_{A(\Delta_f)}$ is the normalized Haar measure on $A(\Delta_f)$. 
	Then $\mu_{\Delta_f}$ is the push-forward measure of $\omega_{A(\Delta_f)}$ through
	$$\int_{\bR}h(t)\diff\mu_{\Delta_f}(t) = \int_{A(\Delta_f)}h\Big(m_0(\chi_f) + \tfrac12+ (m_{\pi}(\chi_f) + \tfrac12)e^{ia_0} + 2 \sum_{j=1}^{r}m_{\theta_j}(\chi_f)\cos(a_j)\Big)\diff\omega_{A(\Delta_f)}(a)$$ for any bounded continuous functions $h : \bR \to \bR$.
	
	Now, $\mu_{\Delta_f}$ is symmetric with respect to $m_0(\chi_f)+ \frac12$ if and only if, for every continuous function $h$, one has
	$$\int_{\mathbb{R}}h(2m_0(\chi_f)+1 - t)\diff\mu_{\Delta_f}(t) = \int_{\mathbb{R}}h(t)\diff\mu_{\Delta_f}(t). $$
Observe that
\begin{align*}
\int_{\mathbb{R}}h(2m_0(\chi_f) +1 - t)\diff\mu_{\Delta_f}(t) 
&= \int_{A(\Delta_f)}h\Big(m_0(\chi_f) + \tfrac12 - (m_{\pi}(\chi_f)+ \tfrac12)e^{ia_0} - 2\sum_{j=1}^{r}m_{\theta_j}(\chi_f) \cos(a_j)\Big)\diff\omega_{A(\Delta_f)}(a) \\
&= \int_{A(\Delta_f)}h\Big(m_0(\chi_f) + \tfrac12 +(m_{\pi}(\chi_f) + \tfrac12)e^{i(a_0 + \pi)} +  2\sum_{j=1}^{r}m_{\theta_j}(\chi_f) \cos(a_j+\pi)\Big)\diff\omega_{A(\Delta_f)}(a).
\end{align*}
So, if $(\pi,\dots,\pi) \in {A(\Delta_f)}$, using the fact that the Haar measure is translation-invariant, we deduce that $\mu_{\Delta_f}$ is symmetric with respect to $m_0(\chi_f) + \frac12$.

On the other hand, assume $(\pi,\dots,\pi) \notin A(\Delta_f)$. Then as $A(\Delta_f)$ is closed, and $m_{\pi}(\chi_f), m_{\theta_j}(\chi_f) \geq 0$ there exists $\epsilon >0$ such that for each $a \in A(\Delta_f)$ one has\footnote{See Lemma~\ref{Lemma: asymmetric} for an explicit bound.}
$$  (m_{\pi}(\chi_f)+\tfrac12)e^{ia_0} + 2\sum_{j=1}^{r}m_{\theta_j}(\chi_f) \cos(a_j) \geq \epsilon - m_{\pi}(\chi_f) - \tfrac12 -  2\sum_{j=1}^{r}m_{\theta_j}(\chi_f) .$$
Let $h_{\epsilon}$ be a non-zero, non-negative function, supported in an interval of length $\epsilon$ around $m_0(\chi_f) - m_{\pi}(\chi_f) - 2 \sum_{j=1}^{r}m_{\theta_j}(\chi_f) $.
Then 
$$
\int_{\mathbb{R}}h_{\epsilon}(t)\diff\mu_{\Delta_f}(t) = 0
$$
while 
$$
\int_{\mathbb{R}}h_{\epsilon}(2m_0(\chi_f) -t)\diff\mu_{\Delta_f}(t) > 0.
$$
In particular, we deduce that $\mu_{\Delta_f}$ is not symmetric with respect to $m_0(\chi_f)$.
\end{proof}

\subsection{Results about linear recurrence sequences}
\label{recseq}

We are interested in the positivity and zero-sets of the quantities $\Delta_f(n)$ defined in \ref{eq:Delta}. One of the key insight is that those quantities are linear recurrence sequences which will imply the limits in Definitions \ref{complete}, \ref{lower}, and \ref{wrongdirection} exist as shown in Corollaries \ref{Cor : density lower order}, \ref{cor:complete}.

\begin{definition}
A linear recurrence sequence of order $k \in \mathbb Z_{> 0}$ is a sequence $(a_n)_{n \in \N}$ such that there exist $u_0, \dots, u_{k-1} \in \bC$ satisfying $$a_{n+k} = u_{k-1} a_{n+k-1} + \dots + u_0 a_n$$ for all $n \in \N$. We define its zero-set as $\{n \in \mathbb Z_{> 0} \mid a_n = 0\}$.
\end{definition}

It is classical that any linear recurrence sequence can be expressed in a generalized power sum form and that, conversely, any generalized power sum satisfies a linear recurrence relation.
\begin{lemma} \label{Lemma:rec seq}
Let $f \in \mathcal{H}_n(\F_q)$, then the sequence $\Delta_f$ is a linear recurrence sequence. 
\end{lemma}

\begin{proof}
    Let $P_f$ be the reversed zeta function of the curve $\mathcal C_f : y^2 = f(x)$ with $f \in \mathcal{H}_n(\F_q)$, and let $\chi_f$ be the primitive quadratic character modulo $f$ and $g$ be the genus of $\mathcal{C}_f$. The roots of $P_f$ are $\alpha_1, \dots, \alpha_{2g}$ which are of the form $\sqrt q e^{i \theta_i(\chi_f)}$ with some of them possibly $\pm \sqrt q$. Then, the conclusion follows from \cite{RecSeq}*{page~3}.
\end{proof}

Note that Lemma~\ref{Lemma:rec seq} is a well-known fact that follows directly from the rationality of the $L$-function. It is not a particularity of hyperelliptic curves. We stated and proved the result here, as this is the first time it is used in the context of studying Chebyshev's bias.

It turns out one can characterize the zero-set of such a linear recurrence sequence following the Skolem-Mahler-Lech theorem, which is stated below. A very short proof over $\Q$ (the Skolem case), which is the case of interest for us, using $p$-adic analysis, is given in \cite{RecSeq}*{Theorem 2.1}.

\begin{theorem}[Skolem-Mahler-Lech, \cite{RecSeq}*{Theorem 2.1}] \label{Thm:Sko}
Assume $(a_n)_{n \in \mathbb{N}}$ is a linear recurrence sequence over a field of characteristic zero. Then its zero-set is the union of a finite set and a finite number of arithmetic progressions.
\end{theorem}

This allows us to show that the density in the Definition \ref{lower} of a lower order bias always exists.

\begin{corollary}\label{Cor : density lower order}
The density $\mathrm{dens}(\Delta_f(n)=0)$ in Definition \ref{lower} for \textit{lower order bias} exists.
\end{corollary}

\begin{proof}
By Lemma \ref{Lemma:rec seq}, $\Delta_f$ is a linear recurrence sequence. Its zero-set is a finite union of arithmetic progressions and a finite set following Theorem \ref{Thm:Sko}, therefore it admits a natural density.
\end{proof}

Another useful fact is the following result, showing that the densities considered for complete biases and reversed biases exist.

\begin{theorem}[\cite{BeGer}*{Theorem 1}]
Let $(a_n)_{n \in \mathbb N}$ be a linear recurrence sequence of real numbers. Then its positivity set $\{n \in \mathbb N \mid a_n > 0\}$ admits a natural density.
\end{theorem}

\begin{corollary} \label{cor:complete}
The densities $\mathrm{dens}(\Delta_f>0)$ and $\mathrm{dens}(\Delta_f<0)$ in Definitions \ref{complete} and \ref{wrongdirection} exist.
\end{corollary}

For certain kinds of linear recurrence sequences, called non-degenerate linear recurrence sequences, we know their zero-sets are finite. We introduce the following more general terminology for the character $\chi_f$ inspired by \cite{RecSeq}*{Section 1.1.9} because it will be an important condition to study in the proofs of \eqref{item lower order th main} and \eqref{item wrong direction th main} in Theorem \ref{thm:main}.

\begin{definition}\label{nondegenerate}
We say that $\chi_f$ is non-degenerate when none of $\frac{\alpha_i}{\alpha_j}$, for $1 \leq i \neq j \leq r$, and none of $\frac{\overline{\alpha_i}}{\alpha_j}$, for $1 \leq i, j \leq r$, is a root of unity.
\end{definition}

Using Definition \ref{nondegenerate}, we prove the following Lemma which will be of important use in the study of lower order bias in Section \ref{sec : lower order}.

\begin{lemma}\label{lem : non degenerate lower order bias}
Assume $\chi_f$ is non-degenerate as in Definition \ref{nondegenerate}. Then the zero-set of $\Delta_f(n)$ is finite.
\end{lemma}

\begin{proof}
By \cite{RecSeq}*{page 25}, a non-degenerate linear recurrence sequence, that is, a sequence whose characteristic roots $\beta_1, \dots, \beta_d$ satisfy that no $\frac{\beta_i}{\beta_j}$ is a root of unity for $i \neq j$, takes a given value only finitely many times. In our case however, the characteristic roots are $\frac{\alpha_1}{\sqrt q}, \dots, \frac{\alpha_r}{\sqrt q}, \frac{\overline{\alpha_1}}{\sqrt q}, \dots, \frac{\overline{\alpha_r}}{\sqrt q}$, but also $1$ and $-1$ because of the terms $m_0(\chi_f) + \frac{1}{2}$ and $\left(m_{\pi}(\chi_f) + \frac{1}{2}\right)(-1)^n$ in $\Delta_f(n)$, and obviously $\frac{1}{-1}$ is a root of unity. But it is easily seen that $\Big(\Delta_f(2n) - \left(m_0(\chi_f) + \frac{1}{2}\right) - \left(m_{\pi}(\chi_f) + \frac{1}{2}\right)\Big)_{n\geq 0}$ and $\Big(\Delta_f(2n+1) - \left(m_0(\chi_f) + \frac{1}{2}\right) +\left(m_{\pi}(\chi_f) + \frac{1}{2}\right)\Big)_{n\geq 0}$ are linear recurrence sequences (\cite{RecSeq}*{Theorem 1.1} and \cite{RecSeq}*{Theorem 1.3}), and when $\chi_f$ is non-degenerate according to Definition \ref{nondegenerate}, then those are non-degenerate as linear recurrence sequences. In particular, they respectively take the values $-\left(m_0(\chi_f) + \frac{1}{2}\right) - \left(m_{\pi}(\chi_f) + \frac{1}{2}\right)$ and $-\left(m_0(\chi_f) + \frac{1}{2}\right) + \left(m_{\pi}(\chi_f) + \frac{1}{2}\right)$ a finite number of times, which proves that $\Delta_f(n)$ vanishes a finite number of times.
\end{proof}

\begin{remark}
In the non-degenerate case, we could replace the densities in Definitions \ref{complete} and \ref{wrongdirection} by the corresponding densities for $\Pi(n; \chi_f)$ since they exist and coincide with the ones about $\Delta_f$ in that case following the fact that the density $\mathrm{dens}(\Delta_f(n)=0)$ in Definition \ref{lower} is zero.
\end{remark}

\subsection{A large sieve statement}
\label{subsec : large sieve}

Let $\CSp_{2g}(\F_\ell)$ be the group of symplectic similitudes\footnote{This is sometimes called the general symplectic group and denoted as $\text{GSp}$}  in $\GL_{2g}(\F_{\ell})$. It contains matrices $M \in \GL_{2g}(\F_{\ell})$ such that there exists a scalar $m \in \F_{\ell}^*$, called the multiplicator of $M$, satisfying $M^{\top} J M = m J$ with $J = \begin{pmatrix}0 & I_g\\ - I_g & 0
\end{pmatrix}.$ When $M$ is a symplectic similitude with multiplicator $m$, we say that $M$ is $m$-symplectic.
In this paper, following \cite{Kowalski book}*{page 158} but with a reversed convention, we call $m$-symplectic, any monic polynomial $P$ of even degree $2g$ satisfying $$P(T) = m^{-g}T^{2g}P\left(\frac{m}{T}\right).$$ In particular, for $f \in \mathcal H_n(\mathbb F_q)$ the polynomial $P_f$ as defined in \eqref{eq def P_f} is $q$-symplectic.

Let us first state the result Theorem \ref{Th Frobenius large sieve new version} for a general setting, using Perret-Gentil's improvement of Kowalski's large sieve for Frobenius \cite{Perret-GentilANT}*{Theorem~5.14.(ii).(c)} and later apply it to our setting in Proposition \ref{Prop Frobenius large sieve}.

The theorem is given for a general $U/\mathbb{F}_p$ smooth affine geometrically connected algebraic variety of dimension $d \geq 1$ over $\mathbb{F}_p$. We assume that  $U$ has a compactification where it is the complement of a divisor with normal crossing.
We denote by $\bar{\eta}$  a geometric generic point of $U$.

Let us fix $\Lambda$ a set of primes different from $2$ and $p$ of density $1$. We study a family $\mathcal{F}_{\ell}$ of lisse sheaves  of $\mathbb{F}_{\ell}$-vector spaces on $U$, corresponding to continuous homomorphisms
$\rho_{\ell} : \pi_1 (U, \bar{\eta}) \rightarrow \GL_r(\mathbb{F}_{\ell}  )$, for $\ell \in \Lambda$ that arise from a compatible system (as in \cite{Kowalski book}*{Definition~8.7}).
Then for $\ell \in \Lambda$,  we denote $G_{\ell} = \rho_{\ell}(\pi_1 (U, \bar{\eta}))$ and $G_{\ell}^{\mathrm{geo}} = \rho_{\ell}(\pi_1 (U_{\overline{\mathbb{F}}_q}, \bar{\eta}))$.

\begin{theorem}
\label{Th Frobenius large sieve new version}
Let $p$ be a prime number and $q>1$ be a power of $p$.
For each $\ell \in \Lambda$ fix $\Omega_\ell \subset G_{\ell}$ a conjugacy invariant subset, in the coset  $\rho_{\ell}(\frob_{f,q})G_{\ell}^{\mathrm{geo}}$.

Then, for any $L\geq 1$ and for any $q$ which is a power of $p$, one has
\begin{align*}
    \frac{\lvert \lbrace f \in U(\F_q) \mid \rho_{\ell}(\frob_{f,q}) \notin \Omega_{\ell} \text{ for all } \ell \leq L, \ell\in \Lambda \rbrace\rvert}{\lvert U(\F_q)\rvert} \leq \Big(1 + \frac{(L+1)^{A} C}{q^{\frac12}} \Big) H^{-1},
\end{align*}
with $C=C(U_{\overline{\mathbb{F}}_q},\lbrace\rho_\ell\rbrace_{\ell \in \Lambda})$ a constant that depends only on $U_{\overline{\mathbb{F}}_q}$ and on the family $\lbrace\rho_\ell\rbrace_{\ell \in \Lambda}$ (in particular not on $q$, but certainly on $d$),
\begin{align}\label{H}
    H = \sum_{\substack{m \in \mathcal{L} \\ \psi(m)\leq L+1}} \prod_{\ell \mid m} \frac{\lvert \Omega_\ell \rvert}{\lvert G_{\ell}^{\mathrm{geo}} \rvert - \lvert \Omega_{\ell}\rvert},
\end{align}
where $\mathcal{L}$ is the set of squarefree integers whose prime factors are all in $\Lambda$, $\psi(m) := \prod_{\ell\mid m}(\ell +1)$,
and when $ G_{\ell}^{\mathrm{geo}} = \Sp(2g,\mathbb{F}_{\ell})$ one can take $A = 2g^2 + g +2$. 
\end{theorem}

\begin{proof}
We are in the setting of \cite{Kowalski book}*{Chapter 8}, following the ideas and notations of loc. cit. It follows from \cite{Kowalski book}*{Proposition~2.9} as in \cite{Kowalski book}*{Corollary~8.10} that
$$\#\lbrace f \in U(\F_q) \mid \rho_{\ell}(\frob_{f,q}) \notin \Omega_{\ell} \text{ for all } \ell \leq L, \ell\in \Lambda \rbrace
\leq \Delta H^{-1},$$
where $H$ is as defined in \eqref{H} and
$\Delta$ is the large sieve constant.
As in the proof of \cite{Kowalski book}*{Proposition~8.8} we obtain that
$$\Delta \leq \max_{\substack{m\in \mathcal{L}\\ \psi(m)\leq L+1}}\max_{\pi \in \Pi_m^*}\sum_{\substack{n\in \mathcal{L}\\ \psi(n)\leq L+1}}\sum_{\tau \in \Pi_n^*} \lvert W(\pi,\tau)\rvert$$
with 
$$W(\pi,\tau) = \delta((m,\pi),(n,\tau))q^{d} + O(\sigma'_{c}(\bar{U},\mathcal{W}(\pi,\tau))q^{d-\frac12})$$ 
where $\mathcal{W}(\pi,\tau)$ is the lisse sheaf corresponding to the representation $[\pi,\bar{\tau}]$ as defined in \cite{Kowalski book}*{(3.8)}, and $\sigma'_{c}$ is the sum of all except the largest Betti numbers as defined in  \cite{Kowalski book}*{page~166}.
In \cite{Perret-GentilANT}*{Section~5D2},  Perret-Gentil improves the bound on  $\sigma'_{c}(U_{\overline{\mathbb{F}}_q},\mathcal{W}(\pi,\tau))$ compared to the bound of \cite{Kowalski book}*{Proposition~8.8} in the case of the complement of a divisor with normal crossing. He obtains
$$\sigma'_{c}(U_{\overline{\mathbb{F}}_q},\mathcal{W}(\pi,\tau)) \ll_{U,\rho} \dim [\pi,\bar{\tau}] = \dim \pi \dim \tau, $$
where the implicit constant depends on $U_{\overline{\mathbb{F}}_q}$ and on the family $\lbrace\rho_\ell\rbrace_{\ell \in \Lambda}$ (in particular not on $q$, but certainly on $d$ and on $p$).
Thus, we have
\begin{align*}
    \Delta \leq q^d + q^{d-\frac12}C(U,\rho)\max_{\substack{m\in \mathcal{L}\\ \psi(m)\leq L+1}}\max_{\pi \in \Pi_m^*}\sum_{\substack{n\in \mathcal{L}\\ \psi(n)\leq L+1}}\sum_{\tau \in \Pi_n^*} \dim \pi \dim \tau. 
\end{align*}
To conclude, we use \cite{Kowalski book}*{(8.13)}, and multiplicativity. 
In particular, representations of $\Sp(2g,\mathbb{F}_{\ell})$ satisfy $\dim \pi \leq (\ell +1)^{g^2}$ and  $\sum_{\pi \in \Pi_{\ell}^*} \dim \pi \leq (\ell +1)^{g^2 +g +1}$.
\end{proof}

To improve on Kowalski's bound \eqref{Kowalski} in Theorem \ref{thm:Kowalski}, we are going to use the following large sieve result which follows from Theorem~\ref{Th Frobenius large sieve new version} applied to the variety of configurations, with the compatible system given by the action of the Frobenius.

\begin{proposition}\label{Prop Frobenius large sieve}
Let $p$ be a prime number and $q>1$ be a power of $p$. 
Let $n\geq 2$, $\mathscr{H}_n$ be the configuration space of monic squarefree polynomials of degree $n$ and $\Lambda$ be the set of primes different from $2$ and $p$.

For each $\ell \in \Lambda$, the action of the Frobenius endomorphism $\frob_{f,q}$ on $\textup{H}^1_{\text{\'et}}(\mathcal{C}_f,\mathbb{Z}_\ell)$ gives a representation $\rho_\ell : \pi_1(\mathscr{H}_n,\bar\eta)\rightarrow \GL_{2g}(\F_\ell)$
for $\bar\eta$ a geometric generic point and for all $\ell \in \Lambda$ they form a compatible system (as in \cite{Kowalski book}*{Definition~8.7}), with image equal to the set of $q$-symplectic similitudes following the work of Hall \cite{Hall}.

For every $\ell \in \Lambda$, let $\Omega_\ell \subset \CSp_{2g}(\F_\ell)$ be a conjugacy invariant subset such that the multiplicator of every element of $\Omega_\ell$ is $q$.

Then, one has
\begin{align*}
    \frac{\#\lbrace f \in \mathcal{H}_n(\F_q) \mid \rho_{\ell}(\frob_{f,q}) \notin \Omega_{\ell} \text{ for all } \ell < q^{\frac{1}{2A}}, \ell\in \Lambda \rbrace}{\lvert \mathcal{H}_n(\F_q)\rvert} \ll_{p, n} \    \Big( \sum_{\substack{\psi(m)\leq q^{\frac1{2A}} \\ m \in \mathcal{L}}} \prod_{\ell \mid m}\frac{\lvert \Omega_\ell \rvert}{\lvert \Sp_{2g}(\F_\ell)\rvert - \lvert \Omega_{\ell}\rvert}\Big)^{-1},
\end{align*}
where the implicit constant depends only on $n$ and $p$, we can take $A = 2g^2 + g+2$, $\mathcal{L}$ is the set of squarefree integers whose prime factors are all in $\Lambda$, and $\psi(m) = \prod_{\ell \mid m}(\ell +1)$.
\end{proposition}

\begin{proof}
We are in the setting of Theorem~\ref{Th Frobenius large sieve new version} with  $U= \mathscr{H}_n$  of dimension $n \geq 2$.
The variety $\mathscr{H}_n \subset \mathbb{A}^n$ is defined by the non-vanishing of the discriminant, it is thus a smooth affine geometrically connected algebraic variety which is the complement of a divisor with normal crossing (\cite{EVW}*{Lemma~7.6}).

As in \cite{Kowalski book}*{Section~8.6} for each $\ell \neq 2,p$, the sheaf
$\mathcal{F}_\ell$ corresponding to $\rho_\ell$ is a rank $2g$ lisse sheaf of $\F_\ell$-modules on $\mathscr{H}_n$.
Since the action of the Frobenius on $H^1(C,\mathbb{Z}_\ell)$ is independent of $\ell$, the representations $\rho_\ell$   arise from a compatible system.
By \cite{Hall}*{Theorem 1.2} (attributed to Yu), 
  the images of $\pi_1(\mathscr{H}_n,\bar\eta)$ and of $\pi_1(\overline{\mathscr{H}}_n,\bar\eta)$ (arithmetic and geometric monodromy groups) are conjugate to $\Sp_{2g}(\F_\ell)$ for all $\ell \neq 2,p$.

Hence, the bound follows from Theorem~\ref{Th Frobenius large sieve new version}, where we chose $L +1 = q^{\frac1{2A}}$.
\end{proof}

\begin{remark} \label{rk : change Lambda}
Note that for any finite set of primes $S$, the result of Proposition \ref{Prop Frobenius large sieve} holds with the set $\Lambda$ replaced by $\Lambda' = \Lambda \setminus S$, and the set $\mathcal{L}$ replaced by the set $\mathcal{L}'$ of squarefree integers with prime factors in $\Lambda'$. This is used in the proof of Lemma \ref{Lemma bound wrong direction condition}.
\end{remark}

\section{Linear dependence}
\label{sec : LI}

Kowalski's Theorem \ref{thm:Kowalski} is concerned with one-parameter families of reducible squarefree polynomials. The large sieve result Proposition \ref{Prop Frobenius large sieve} above allows us, following Kowalski's proof in \cite{Kowalskibook}, to get the exact same bound, but for the larger space of parameters $\mathcal{H}_n(\F_q)$.

\begin{proof}[Proof of Theorem~\ref{thm:main}.\ref{item LI main}.]
We follow exactly the proof of \cite{Kowalski book}*{Theorem 8.15} but instead of using \cite{Kowalski book}*{Corollary 8.10}, we use Proposition \ref{Prop Frobenius large sieve}. 
Thus, we obtain 
\begin{equation*}
\frac{1}{|\mathcal{H}_n(\F_q)|}    \#\{f \in \mathcal{H}_n(\F_q) \mid \text{ The zeta function of } \mathcal{C}_f \text{ does not satisfy } \LI\} \ll_{p, g} H_1^{-1} + H_2^{-1} + H_3^{-1} + H_4^{-1},
\end{equation*}
where for $i = 1, \dots 4$, $$H_i =  \sum_{\substack{\psi(m)\leq q^{\frac1{2A}} \\ m \in \mathcal{L}}} \prod_{\ell \mid m}\frac{\lvert \Omega_{i,\ell} \rvert}{\lvert \Sp_{2g}(\F_\ell)\rvert - \lvert \Omega_{i,\ell}\rvert},$$
and the sets $\Omega_{i,\ell}$ are defined as in \cite{Kowalski book}*{pages 179--180}.
In particular, 
\begin{enumerate}
    \item $\Omega_{1,\ell}$ is the set of matrices  $M \in \CSp_{2g}(\mathbb{F}_{\ell})$ with multiplicator $q$ such that $\chi_M (X)$ is irreducible, and  \cite{Kowalski book}*{page 181} gives $\frac{\lvert \Omega_{1,\ell}\rvert}{\lvert \Sp(\F_\ell) \rvert} \geq \frac1{2g}$.
    \item $\Omega_{2,\ell}$ is the set of matrices  $M \in \CSp_{2g}(\mathbb{F}_{\ell})$ with multiplicator $q$ such that $\chi_M (X)$ factors as a product of an irreducible quadratic polynomial and a product of irreducible polynomials of odd degree, which satisfy\footnote{a factor $\frac12$ was forgotten in \cite{Kowalski book}*{page 181}.} $\frac{\lvert \Omega_{2,\ell}\rvert}{\lvert \Sp(\F_\ell) \rvert} \geq \frac1{8g}$ by Lemma~\ref{lemma counting polynomials} (with $k=1$, $n_0=1$, $n_{\frac{g-3}{2}}=1$ in the case $g$ is odd) and \cite{Kowalski book}*{Lemma~B.5}.
    \item $\Omega_{3,\ell}$ is the set of matrices  $M \in \CSp_{2g}(\mathbb{F}_{\ell})$ with multiplicator $q$ such that 
    the polynomial $h$ defined by  $\chi_M (X) = X^g h(X + q X^{-1})$ factors as a product of  an irreducible quadratic polynomial and a product of irreducible polynomials of odd degree, and  \cite{Kowalski book}*{page 181} gives $\frac{\lvert \Omega_{3,\ell}\rvert}{\lvert \Sp(\F_\ell) \rvert} \underset{g \to +\infty}{\sim} \frac{\log 2}{\log g}$.
    \item $\Omega_{4,\ell}$ is the set of matrices  $M \in \CSp_{2g}(\mathbb{F}_{\ell})$ with multiplicator $q$ such that 
    the polynomial $h$ defined by  $\chi_M (X) = X^g h(X + q X^{-1})$ has an irreducible factor of prime degree $> \tfrac{g}2$,
    and  \cite{Kowalski book}*{page 181} gives $\frac{\lvert \Omega_{4,\ell}\rvert}{\lvert \Sp(\F_\ell) \rvert} \underset{g \to +\infty}{\sim} \frac1{\sqrt{2\pi g}}$.
\end{enumerate}
The final bound is the same (correcting $\delta_2 \geq (4g)^{-1}$ into $\delta_2 \geq (8g)^{-1}$), but the space of parameters $\mathcal{H}_n(\F_q)$ is larger. The dependency on $p$ is lost in the proof of Theorem \ref{Th Frobenius large sieve new version}.\end{proof}

To prove Theorem \ref{main:small genus} for the genus $2$ case, we will use the following result of Ahmadi and Shparlinski.

\begin{theorem}[\cite{AhmadiShpar}*{Theorem 2}] \label{thm:Ahmadi Shparlinski}
Let $\mathcal{C}$ be a smooth projective curve of genus $2$. If the Jacobian of $\mathcal{C}$ is absolutely simple, then the zeta function of $\mathcal{C}$ satisfies $\LI$.
\end{theorem}

\begin{proof}[Proof of Theorem~\ref{main:small genus}.]
Let us first prove the bound when $g=1$ and assume for now that $\deg f = 3$. Then $\mathcal{C}_f$ is an elliptic curve, with two conjugate (possibly equal) Frobenius eigenvalues. The only way for $\LI$ to fail is that those eigenvalues are of the form $\sqrt q \zeta$ with $\zeta$ a root of unity, that is, $\mathcal{C}_f$ has to be a supersingular elliptic curve. By \cite{Silverman}*{V~Theorem 4.1.(c)}, there are $\ll p$ such curves over $\F_q$, up to $\overline{\F}_q$-isomorphism (recall that $q$ is a power of the prime number $p$). But two elliptic curves are isomorphic over $\overline{\F}_q$ if and only if they have the same $j$-invariant (\cite{Silverman}*{III~Proposition 1.4.(b)} which holds in every characteristic). Let $E$ be a fixed supersingular elliptic curve defined over $\F_q$ with $j$-invariant $j$, and let us write $j(f)$ the $j$-invariant of the elliptic curve $\mathcal{C}_f$. Then clearly $j(f)=j$ is a non-zero polynomial equation in the $\deg f$ coefficients of $f$ by the definition of the $j$-invariant \cite{Silverman}*{page~42}. It is indeed non-zero since there always exist a non-supersingular elliptic curve over $\F_q$ (\cite{Waterhouse}*{Theorem 4.1}). In particular, one has $$\#\{(a, b, c) \in \F_q^3 \mid f=x^3+ax^2+bx+c, \mathcal{C}_f \text{ is isomorphic to } E \text{ over } \overline{\F}_q\} \ll q^2.$$ This yields $$\#\{f \in \mathcal{H}_3(\F_q) \mid \mathcal{C}_f \text{ is supersingular }\} \ll pq^2$$ and the result follows since in general $|\mathcal H_n(\F_q)| = q^n - q^{n-1}$. In the case where $\deg f = 4$ we assume that $p \neq 2, 3$. Then $\mathcal{C}_f$ is isomorphic to its Jacobian $J_f$, and by \cite{Cremona}*{page~82}, $J_f$ is given as the smooth projective model of the curve defined by the equation $y^2 = x^3 - 27Ix - 27J$, and $I$ and $J$ are the quartic invariants defined in \cite{Cremona}*{pages~72--73}. The $j$-invariant of $J_f$ is then clearly a non-constant rational function in the coefficients of~$f$, and we conclude as in the case $\deg f = 3$.

Assume now that $g=2$. By Theorem \ref{thm:Ahmadi Shparlinski}, if $\LI$ fails for the zeta function of $\mathcal{C}_f$, then its Jacobian $J_f$ is not absolutely simple, \textit{i.e.} it splits over a finite extension $\mathbb K$ of $\F_q$. In particular, the Weil polynomial $W_{f, \mathbb K}$ of $J_f/\mathbb K$ is reducible. Calling $d$ the degree $[\mathbb K : \F_q]$, one has $W_{f, \mathbb K}(X^d) = \prod_{k=0}^{d-1} W_f(\zeta_d ^k X) = \prod_{k=0}^{d-1} P_f(\zeta_d ^k X)$, where $W_f$ is the Weil polynomial of $J_f/\F_q$, which is equal to $P_f$ (\cite{CorSil}*{VII. Corollary 11.4}), and $\zeta_d$ is a primitive $d$-th root of unity. It easily implies that $W_{f, \mathbb K}$ has roots $\alpha_j(\chi_f)^d$, $j \in \{1, \dots, 4\}$. Now, there are two possible cases. Either one of $\alpha_i(\chi_f)^d$ is a rational number (necessarily $\pm q^{d/2}$), or there are two indices $i \neq j\in \{1, \dots, 4\}$ such that $\alpha_i(\chi_f)^d \alpha_j(\chi_f)^d$ is a rational number (necessarily $\pm q^d$). In particular, $\chi_f$ is degenerate according to Definition \ref{nondegenerate}. We conclude by Lemma~\ref{Lemma bound quotient}.
\end{proof}

\section{ Complete biases}
\label{sec :complete}

\subsection{Upper bounds for  complete biases}\label{subsec:completeconditions}

To derive a necessary condition for exhibiting a complete bias, we will use the following simple inequality of Bhatia and Davis \cite{BhatiaDavis}*{Theorem 1} (the proof in \cite{BhatiaDavis} is done for discrete random variables, but the general case works exactly the same).

\begin{theorem}[Bhatia-Davis Inequality]\label{BD}
Let $X$ be a bounded random variable such that $a \leq X \leq b$ almost-surely with mean $\mu$ and variance $\sigma^2$, then 
\begin{equation}\label{eq:B-D}
    \sigma^2 \leq (b-\mu)(\mu-a).
\end{equation}
\end{theorem}

\begin{proposition}[Necessary condition for complete bias]\label{lemma:necessary-complete}
Let $f\in \F_q[x]$ and assume that $\Pi(n; \chi_f)$ admits a complete bias.
Then one of the following assertions is true.
\begin{enumerate}
    \item The distribution $\mu_{\Delta_f}$ is symmetric with respect to its mean value and $m_0(\chi_f) \geq m_{\pi}(\chi_f) + 2\sum_{j=1}^r m_j(\chi_f)$ and in the case $r=0$, the inequality is strict with more than half of the zeros equal to $\sqrt{q}$.
    \item The distribution $\mu_{\Delta_f}$ is not symmetric with respect to its mean value and $m_0(\chi_f) > m_{\pi}(\chi_f)$.
\end{enumerate}

\end{proposition}
In particular, this implies the following condition.
\begin{corollary}\label{cor:necessary cond complete}
    If $\Pi(n; \chi_f)$ admits a complete bias for $f \in \mathbb{F}_q[x]$,
    then  $q$ is a square and $L(\frac{1}{2},\chi_f)=0$.
\end{corollary}

\begin{remark}
In the case of Dirichlet $L$-functions over $\Q$, it is a famous conjecture of Chowla\cite{Chowla} that no such $L$-function can vanish at $\tfrac12$. It is known that Artin $L$-functions corresponding to number fields extensions can vanish at $\tfrac12$. Incidentally, this was used in \cite{Bailleul1} to provide examples of reversed bias in this context. In the function field case, it was shown in \cite{Li18}*{Theorem 1.3} that for any $q$ there are infinitely many Dirichlet $L$-functions over $\F_q(x)$ vanishing at $\tfrac12$, that is such that the corresponding Weil polynomial vanishes at $\sqrt q$. However it is expected that $100\%$ of those $L$-functions do not vanish at~$\tfrac12$ for a fixed $q$ (\cite{Li18}*{Remark 1.4}). If this were true, we would obtain the following result instead of Theorem \ref{thm:complete arithm geom}: for every $q$ a power of an odd prime, $$\lim_{n \to +\infty} \frac{\#\{f \in \mathcal{H}_n(\F_q) \mid \Pi(n; \chi_f) \text{ exhibits a complete bias}\}}{|\mathcal{H}_n(\F_q)|} = 0.$$
Note also that by \cite{ELS}*{Corollary~1.6} there is no complete bias when $f$ is irreducible and $4$ does not divide the degree of $f$. Indeed, in this case $L(\frac 1 2, \chi_f) \neq 0$.
\end{remark}

We can now prove our main results concerning upper bounds for complete bias using the necessary condition in Corollary~\ref{cor:necessary cond complete}.
\begin{proof}[Proof of Theorem~\ref{thm:main}.\ref{item complete bias th main}.]
 The proof follows from applying Corollary~\ref{cor:necessary cond complete}
and Lemma~\ref{Lemma vanish 0}.   
\end{proof}

\begin{proof}[Proof of Theorem~\ref{thm:complete arithm geom}]

By \cite{ELS}*{Theorem 3.2}, one has $$\sup_n \frac{\#\{f \in \mathcal{H}_n(\F_q) \mid \mathcal{L}(\sqrt q, \chi_f) = 0\}}{|\mathcal{H}_n(\F_q)|} \ll q^{-\frac{1}{276}},$$ and so the bound follows from Corollary~\ref{cor:necessary cond complete}.
\end{proof}

We finally give the proof of our necessary condition for complete bias.
\begin{proof}[Proof of Proposition~\ref{lemma:necessary-complete}]
Suppose that the distribution $\mu_{\Delta_f}$ is symmetric with respect to its mean value $m_0(\chi_f) + \frac12$.
We have $\Delta_f(0) = m_0(\chi_f) + m_{\pi}(\chi_f) + 1 + 2\sum_{j=1}^r m_j(\chi_f)$, so this value is in $\mathrm{supp}\mu_{\Delta_f}$. Indeed, let $\varepsilon > 0$ and $h : \bR \to \bR$ be non-negative continuous and supported on $[\Delta_f(0) - \varepsilon, \Delta_f(0) + \varepsilon]$, with $h(\Delta_f(0))  > 0$. 
Then $$\int_{\bR}h(t)\diff\mu_{\Delta_f}(t) = \int_{A(\Delta_f)}\tilde{h}(a_0, \dots, a_r)\diff\omega_{A(\Delta_f)}(a)$$ where $\tilde{h}(a_0, \dots, a_r) = h\Big(m_0(\chi_f) + \tfrac12+ (m_{\pi}(\chi_f) + \tfrac12)e^{ia_0} + 2 \sum_{j=1}^{r}m_{\theta_j}(\chi_f)\cos(a_j)\Big)$ and $\diff\omega_{A(\Delta_f)}$ is the Haar measure on the subtorus $A(\Delta_f)$ of $\mathbb{T}^{r+1}$ generated by $(\pi, \theta_1, \dots, \theta_r)$. Since $h(\Delta_f(0)) = \tilde{h}(0, \dots, 0) > 0$, we get $\int_{\bR}h(t)\diff\mu_{\Delta_f}(t) > 0$, which implies $\Delta_f(0) \in \mathrm{supp}\mu_{\Delta_f}$.

By symmetry, $2(m_0(\chi_f) + \frac12) -(m_0(\chi_f) + m_{\pi}(\chi_f) + 1 + 2\sum_{j=1}^r m_j(\chi_f)) $ is also in $\mathrm{supp}\mu_{\Delta_f}$, so it is non-negative.

In the case $\mu_{\Delta_f}$ is not symmetric with respect to its mean value, we are interested in the behavior of $$\Delta_f : n \mapsto m_0(\chi_f) + \frac{1}{2} + (-1)^n \left(m_{\pi}(\chi_f) + \frac{1}{2}\right) + 2\sum_{j=1}^r m_j(\chi_f) \cos(n\theta_j).$$

By \cite{Bailleul2}*{Theorem~3.1}, we have $\text{dens}(\Delta_f > 0) \leq \frac12 \mathbb P(Y_0 \geq 0)  + \frac12 \mathbb P(Y_1 \geq 0)$ where $Y_0, Y_1$ are random variables whose distributions are the limiting distributions of $\Delta_f(2 \cdot)$ and $\Delta_f(2 \cdot + 1)$ respectively. Since we are assuming complete bias, then $\text{dens}(\Delta_f > 0) = 1$ yields $\mathbb P(Y_0 \geq 0) =  \mathbb P(Y_1 \geq 0)= 1$.

 We apply the Bhatia-Davis Inequality, Theorem \ref{BD}, to the random variable $Y_1$. To do so, we need the maximum, minimum, mean, and variance of $Y_1$. To understand these, we group the $\theta_j$ by pairs such that $\theta_{j'} = \pi - \theta_j$ when necessary. We have 
 \begin{align*}
     \Delta_f(2n+1) &= m_0(\chi_f) - m_{\pi}(\chi_f) + 2\sum_{j=1}^r m_j(\chi_f) \cos(\theta_j (2n+1)) \\
     &= m_0(\chi_f) - m_{\pi}(\chi_f)  + 2\sum_{j=1}^{r'} m_j'(\chi_f) \cos(\theta_j (2n+1)),
 \end{align*}
where we sum on $\lbrace\theta_1,\dots,\theta_{r'}\rbrace = \lbrace\theta_1,\dots,\theta_{r}\rbrace  \setminus\lbrace \theta_j \mid \exists k\leq j, \theta_j = \pi - \theta_k \rbrace$ (in particular $\frac{\pi}{2} \notin \lbrace\theta_1,\dots,\theta_{r'}\rbrace $),
and we define $m_j'(\chi_f) = m_j(\chi_f) - m_{k(j)}(\chi_f)$ where $\theta_{k(j)} = \pi - \theta_j$ (and $m_{k(j)}(\chi_f)=0$ if such a $\theta_{k(j)}$ does not exist). This grouping of terms was made to simplify the computation of the variance below.
From this expression we deduce 
\begin{align*}
    \mathbb{E}(Y_1) = m_0(\chi_f) - m_{\pi}(\chi_f).
\end{align*}
By the assumption of complete bias, we have $Y_1 \geq 0$ almost-surely. 
By the definition of $Y_1$, we have
\begin{align*}
    Y_1  \leq m_0(\chi_f) - m_{\pi}(\chi_f)  + 2\sum_{j=1}^{r'} |m_j'(\chi_f)| \, \text{almost-surely}
\end{align*}
and
\begin{align*}
    \mathrm{Var}(Y_1) &= \lim_{K\rightarrow\infty}\frac{1}{K} \sum_{x=0}^{K-1} \Big(2 \sum_{j=1}^{r'} m_j'(\chi_f) \cos(\theta_j (2x+1)) \Big)^2 \\
    &= \lim_{K\rightarrow\infty}\frac{4}{K} \sum_{x=0}^{K-1} \Big[\sum_{j=1}^{r'} \big(m_j' (\chi_f) \cos(\theta_j (2x+1))\big)^2 + \sum_{1 \leq j\neq k \leq r'} m_j'(\chi_f) m_k'(\chi_f) \cos(\theta_j (2x+1))\cos(\theta_k (2x+1))\Big] \\
    &= 2 \sum_{j=1}^{r'} m_j'(\chi_f)^2.
\end{align*}

By the Bhatia-Davis inequality (Theorem~\ref{BD}),
we obtain 
$$\text{Var}(Y_1) \leq \big(m_0(\chi_f) - m_{\pi}(\chi_f)  + 2\sum_{j=1}^{r'} |m_j'(\chi_f)|-\mathbb E(Y_1)\big)\big(\mathbb E(Y_1)-0\big).$$ 
 This yields 
 \begin{equation}\label{Eq Bathia Davis applied}
      \sum_{j=1}^{r'} m_j'(\chi_f) ^2 \leq 2\sum_{j=1}^{r'} \lvert m_j'(\chi_f) \rvert \left(m_0(\chi_f) - m_{\pi}(\chi_f) \right).
 \end{equation}
 
If every $m_j'(\chi_f)$ is zero, this means that for every integer $n$, one has $\Delta_f(2n+1) = m_0(\chi_f) - m_{\pi}(\chi_f)$.
Since $\Pi(n; \chi_f)$ exhibits a complete bias, this has to be positive, \textit{i.e.} $m_0(\chi_f) > m_{\pi}(\chi_f)$. If there is at least one non-zero $m_j'(\chi_f)$, the inequality \eqref{Eq Bathia Davis applied} also implies $m_0(\chi_f) > m_{\pi}(\chi_f)$.

Finally, since $\sqrt{q}$ and $-\sqrt{q}$ have distinct multiplicities as roots of $P_f \in \mathbb Z[T]$, those must be rational, hence integers, and so $q$ must be a square.
\end{proof}

\subsection{Examples of complete biases}

In this section, we first give a sufficient condition for a complete bias, in the hope to use it to find examples of instances of such an exceptional behavior.

\begin{lemma}[Sufficient condition for complete bias]\label{lemma:sufficient}
Let $f \in \mathbb F_q[x]$.
Write $$P_f(u) = (u-\sqrt{q})^{m_0} (u+\sqrt{q})^{m_{\pi}} L_{1}(u)L_{2}(u)$$ with $L_2(-u) = L_2(u)$ of maximal degree, $\deg L_i = d_i$.
Assume that one of the following assertions holds,
\begin{enumerate}
    \item\label{case strict sufficient complete} we have $m_0 > m_{\pi} + d_1$ and $m_0 + m_{\pi} + 1> d_1 + d_2$, or
\item we have $m_0 \geq m_{\pi} + d_1$ and $m_0 + m_{\pi} + 1 \geq d_1 + d_2$, and
\begin{enumerate}
    \item\label{case continuous sufficient complete} $L_{1}$ admits a root whose angle is not in $\mathbb{Q}\pi$, or
    \item\label{case asymetric sufficient complete} there exists $k_1, \dots, k_{d_1} \in \mathbb{Z}$ satisfying 
$\sum_{i=1}^{d_1} k_i \theta_i \equiv 0 \pmod{2\pi}$ and $\sum_{i=1}^{d_1} k_i$ is odd, 
where $\theta_1,\dots,\theta_{d_1}$ are the angles of the roots of $L_1$.
\end{enumerate}
\end{enumerate}
Then there is a complete bias with modulus~$f$.
\end{lemma}

One such example is $f= t^4 + 2t^3 + 2t + a^7 \in \F_9[t]$ where $a$ is a generator of $\F_9$ over $\F_3$, in \cite{DevinMeng}*{Example~3} the authors show that  $P_f(u) = (u-3)^2$.

\begin{remark}
More generally, in \cite{DevinMeng}*{Proposition~3.1}, based on Honda--Tate ideas (citing \cite{Waterhouse}*{Theorem~4.1}), one can see that for each $q$ square, there exist $f \in \F_q[x]$ of degree $3$ such that the $L$-function of $\chi_f$ is $(1-\sqrt{q}u)^2$. This gives one example satisfying Lemma~\ref{lemma:sufficient} for each $q$ square. 
\end{remark}

\begin{remark}
Note however that our sufficient condition for a complete bias Lemma \ref{lemma:sufficient} is more restrictive than simply vanishing at $\tfrac{1}{2}$ so we cannot use the lower bound from \cite{Li18}*{Theorem 1.3} to give infinitely many examples of complete bias for a fixed $q$.
\end{remark}

\begin{proof}[Proof of Lemma \ref{lemma:sufficient}]
It suffices to prove that under these conditions, we have $\Delta_f(n) >0$ for almost all $n$, where
$\Delta_f$ is defined in \eqref{eq:Delta}.
We order the zeros of $P_f$ so that the first ones correspond to the zeros of $L_1$, with multiplicities. 
Then, for all $n$ we have
$$\Delta_{f}(2n+1) =  m_0(\chi_f) - m_{\pi}(\chi_f) + \sum_{j=1}^{d_1}  \cos(\theta_j (2n+1))$$
and
$$\Delta_f(2n) = 1 + m_0(\chi_f) + m_{\pi}(\chi_f) + \sum_{j=1}^{d_1 + d_2}  \cos(2\theta_j n). $$
Since $\cos(\theta_j n) \geq -1$ for all $j$ and $n$, the conditions in case~\ref{case strict sufficient complete} imply that $\Delta_f(n) >0$ for all $n$.
In the case the conditions of~\ref{case continuous sufficient complete} are satisfied, we have
$\sum_{j=1}^{d_1}  \cos(\theta_j n) > -d_1$ for almost all $n$, since, up to reordering, we can assume that $\theta_1\notin \mathbb{Q}\pi$ which yields $\cos(\theta_1 n) > -1$ for almost all $n$. This concludes the proof in the case~\ref{case continuous sufficient complete}.
In the case~\ref{case asymetric sufficient complete}, it follows from Lemma~\ref{Lemma: asymmetric} that
 $\sum_{j=1}^{d_1}  \cos(\theta_j n) \geq -d_1 + 1 + \cos(\pi(1-\tfrac{1}{\kappa})) > -d_1$ for all $n$, where $\kappa = \sum_{i=1}^{d_1}\lvert k_i\rvert$ and this concludes the proof.
\end{proof}

We conclude this section by proving a technical lemma that was used in the proof of the sufficient condition (Lemma~\ref{lemma:sufficient}).

\begin{lemma}\label{Lemma: asymmetric}
Let $\gamma_1, \dots, \gamma_N \in (0,\pi)$ and assume that there exists $k_1, \dots, k_N \in \mathbb{Z}$ satisfying 
$\sum_{i=1}^{N} k_i \gamma_i \equiv 0 \pmod{2\pi}$ and $\sum_{i=1}^{N} k_i$ is odd.
Then, for all $\ell \in \mathbb{Z}$, we have 
$\max_{1\leq i \leq N}\lVert \ell \gamma_i - \pi\rVert_{2\pi} \geq \frac{\pi}{ \sum_{i=1}^{N} \lvert k_i\rvert}.$
In particular,
$$ \sum_{1\leq i \leq N}\cos( \ell \gamma_i) \geq -N +1 + \cos\Bigg(\pi \Big(1- \frac{1}{ \sum_{i=1}^{N} \lvert k_i\rvert}\Big)\Bigg).$$
\end{lemma}

\begin{proof}
Recall that 
$\lVert \ell \gamma_i - \pi\rVert_{2\pi} = \min_{n\in \mathbb{Z}} \lvert \ell \gamma_i - (2n+1)\pi \rvert$.
For each $i$, let $n_i \in \mathbb{Z}$ be an integer that satisfies this minimum.
We have
\begin{align*}
    \max_{1\leq i \leq N}\lVert \ell \gamma_i - \pi\rVert_{2\pi} 
    &= \max_{1\leq i \leq N}   \big\lvert \ell  \gamma_i - (2n_i+1)\pi\big\rvert \\
    &\geq \frac{1}{ \sum_{i=1}^{N} \lvert k_i\rvert} \Big\lvert  \ell \sum_{i=1}^{N} k_i \gamma_i - \sum_{i=1}^{N} k_i(2n_i +1) \pi \Big\rvert \\
     &\geq \frac{\pi}{ \sum_{i=1}^{N} \lvert k_i\rvert}.
\end{align*}
Now, suppose that  $\lVert \gamma - \pi\rVert_{2\pi} \geq \frac{\pi}{ \kappa}$,
then we have 
\begin{align*}
\cos(  \gamma) \geq \cos(\pi (1- \tfrac{1}{\kappa})) .
\end{align*}
This concludes the proof.
\end{proof}

\section{ Lower order biases}
\label{sec : lower order}
\subsection{Upper bound}

Our reflections on linear recurrence sequences from Section~\ref{recseq} give a good understanding on lower order bias.
In particular, the contraposition of Lemma~\ref{lem : non degenerate lower order bias} yields the following necessary condition for a lower order bias.

\begin{proposition}[Necessary condition for lower order bias]\label{lem : necessary condition lower order}
    If $\Pi(n; \chi_f)$ admits a lower order bias, then $\chi_f$ is degenerate (see Definition~\ref{nondegenerate}).
\end{proposition}

This lemma implies that for $\Pi(n; \chi_f)$ to admit a lower order bias, the Jacobian of the curve $C_f:y^2=f(x)$ is either non-ordinary or geometrically admitting an isogenous factor of order at least $2$.

Using this lemma and an application of the large sieve from Proposition~\ref{Prop Frobenius large sieve}, we obtain the proof of Theorem~\ref{thm:main}.\ref{item lower order th main}.

\begin{proof}[Proof of Theorem~\ref{thm:main}.\ref{item lower order th main}]
     The proof follows from applying  Proposition~\ref{lem : necessary condition lower order}
and Lemma~\ref{Lemma bound quotient}.   
\end{proof}

\subsection{A sufficient condition for lower order bias and examples}

\begin{lemma}[Sufficient condition for lower order bias]\label{lem : sufficient lower order}
Let $f \in \mathbb F_q[x]$.
Suppose that $P_f(u) = P_f(-u)$,
then $\Delta_f(2n+1)=0$ for all $n$, in particular, there is a lower order bias with modulus $f$.
\end{lemma}

\begin{proof}
Assume that the roots of $P_f$ with positive imaginary parts are labelled so that their arguments are $\theta_1,...,\theta_t, \pi-\theta_1,...,\pi-\theta_t$. Since $P_f(u)=P_f(-u)$, the multiplicity of $\theta_i$ equals to that of $\pi-\theta_i$. For $n \in \mathbb{N}$ and $1 \leq j \leq t$, one has $\cos((\pi-\theta_j)(2n+1)) = -\cos(\theta_j(2n+1))$, whence
\[
\sum_{j=1}^t 2 m_{\theta_j}(\chi_f) \cos(\theta_j(2n+1)) + \sum_{j=1}^t 2 m_{\pi-\theta_j}(\chi_f) \cos((\pi-\theta_j)(2n+1)) = 0.
\]
Further,
\[
\left( \frac{1}{2} + m_0(\chi_f)\right) + \left(\frac{1}{2}+ m_\pi(\chi_f)\right)(-1)^{2n+1}=0.
\]
The above together give  $\Delta_f(2n+1)=0$ for all $n \in \mathbb{N}$. This is sufficient to deduce that there is a lower order bias with modulus $f$.  
\end{proof} 

One such example is $f= t^6 + 2t^3 + 5 \in \F_{23}[t]$ which is irreducible and the $L$-function of $\chi_f$ is $1 - 29u^2 + 23^2u^4$  which is even with $4$ inverse roots $\pm \alpha$, $\pm \overline{\alpha}$, where 
$$\alpha = \sqrt{23}\exp\left(\frac{i}{2}\arctan\left(\left(\frac{5\sqrt{51}}{29}\right)\right)\right).$$
Moreover, using \cite{Calcut}*{page 17}, we see that $\alpha$ has argument unrelated to $\pi$.

\begin{remark}
    Using \cite{HNR}*{Table~1.2} and the sufficient condition, we can give several examples for each $q$ that have a lower order bias. Namely the authors show that the polynomial $X^4 - bX^2 + q^2$ with $b = 2q\cos(2\theta)$ is the Weil polynomial of the Jacobian of a hyperelliptic curve of genus 2 if $b \in \Z$, $b\neq q, 2q, 2q-1, 2q-2$ and $b+2q$ is not a square. Since the Weil polynomial of the Jacobian of such a curve is equal to the corresponding $P_f$ (\cite{CorSil}*{VII. Corollary 11.4}), such $f$ exhibit lower order biases.
\end{remark}

\begin{remark}
The condition of Lemma \ref{lem : sufficient lower order}  gives rise to the following question: 
Fix a finite field $\mathbb{F}_q$,  how many hyperelliptic curves admit even Frobenius characteristic polynomials? If $C$ is such a curve, then $C \otimes \mathbb{F}_{q^2}$ has its Frobenius characteristic polynomial being a perfect square. This question is closely related to counting curves/characters whose $L$-functions are perfect squares. 
\end{remark}

\section{ Reversed biases}
\label{sec : wrong direction}

 \subsection{Upper bound}

Let us first give a necessary condition for a reversed bias.

\begin{proposition}[Necessary condition for a reversed bias] \label{Lemma necessary wrong direction}
If there is a reversed bias with modulus $f$ then 
\begin{itemize}
    \item either there exist $k_1, \dots, k_g \in \mathbf{Z}$ satisfying 
$ \sum_{i=1}^{g} k_i \theta_i \equiv 0 \pmod{2\pi}$ and $\sum_{i=1}^{g} k_i$ is odd, where the $\theta_i$ are angles of zeros of $P_f$.
\item or $m_0 < m_{\pi}$ (in particular, $q$ is a square).
\end{itemize}

\end{proposition}

\begin{proof}
Suppose $f$ admits a reversed bias. Then the distribution $\mu_f$ is not symmetric with respect to its mean value $m_0 + \frac12\geq 0$.
So, from Lemma~\ref{lem : symmetry}, there exists $k_0, k_1, \dots, k_r$ such that $k_0 + \sum_{j=1}^r k_j \equiv 1 \pmod2$ and $k_0\pi + \sum_{j=1}^{r}k_j\theta_j \equiv 0 \pmod{2\pi}$.

If $k_0$ is even, we get the first condition.
Otherwise, assume that all relation between the $\theta_j$'s, $ \sum_{i=1}^{g} k_i \theta_i \equiv 0 \pmod{2\pi}$ satisfy $\sum_{i=1}^{g} k_i$ is even.  Then we deduce from  Lemma~\ref{lem : symmetry}, that the limiting distribution of the functions $\Delta(2\cdot)$ and $\Delta(2\cdot +1)$ are symmetric with respect to their mean values, which are $m_0 + m_\pi +1$ and $m_0 - m_{\pi}$. If the probability that one of the two functions is negative is larger than $\frac12$, then at least one of the mean values has to be negative.
\end{proof}

Here is a translation of our necessary condition in terms of the Galois group of $P_f$ over $\mathbb Q$, which is more convenient to use in the large sieve. Recall that for $f \in \mathcal{H}_n(\F_q)$, $\mathrm{Gal}_{\mathbb Q}(P_f)$ is a subgroup of $W_{2g} = \mathfrak S_g \ltimes \left(\Z/2\Z\right)^g$, itself a subgroup of $\mathfrak S_{2g}$ (see \cite{Kowalskibook}*{page 249}). In the following, we will consider that $\mathrm{Gal}(P_f) \subset \mathfrak S_{2g}$ acts on $\{-g, \dots, -1, 1, \dots, g\}$, the set of indices of the roots $\alpha_1, \dots, \alpha_g, \alpha_{-1} = \overline{\alpha_1}, \dots, \alpha_{-g} = \overline{\alpha_g}$. The fact that $\mathrm{Gal}(P_f) \subset W_{2g}$ means that if $\sigma \in \mathrm{Gal}(P_f)$ then $\sigma(-i)=-\sigma(i)$ for all $i \in \{-g, \dots, -1, 1, \dots, g\}$.

\begin{lemma}\label{Lemma Galois condition WD}
Let $f \in\mathcal{H}_n(\mathbb{F}_q)$. Assume that there exist $k_1, \dots, k_g \in \Z$ such that $k_1 \theta_1 + \dots + k_g \theta_g \equiv 0 \text{ mod } 2\pi$ and $k_1 + \dots + k_g \equiv 1 \pmod{2}$. Then at least one of the following conditions hold:

\begin{enumerate}
    \item $P_f$ is not separable.
    \item $\chi_f$ is degenerate (in the sense of Definition~\ref{nondegenerate}).
    \item $\mathrm{Gal}_{\Q}(P_f)$ does not act transitively on the set of pairs $\lbrace \lbrace 1, -1 \rbrace, \dots, \lbrace g, -g\rbrace\rbrace $;
    \item For every $i \in \{1, \dots, g\}$, $\mathrm{Gal}(P_f)$ does not contain the transposition $(i \, -i)$, and for every pair $\{i, j\}$, with $i \neq j \in \{1, \dots, g\}$, $\mathrm{Gal}_{\Q}(P_f)$ does not contain the $4$-cycle $(i \, j \, -i \, -j)$.
\end{enumerate}

\end{lemma}

\begin{proof}
Assume that none of the first three items are satisfied.
Let us fix $i \in \lbrace 1, \dots, g\}$, then for every $j \in \lbrace 1, \dots, g\} \setminus \{i\}$, there exist $\sigma_{j} \in \mathrm{Gal}(P_f)$ such that $\sigma_j(j) \in \pm i$. From the multiplicative relation $$\left(\frac{\alpha_1}{\sqrt q}\right)^{k_1} \times \dots\times\left(\frac{\alpha_g}{\sqrt q}\right)^{k_g} = 1$$
with $\sum_{j=1}^g k_j \equiv 1 \pmod2$, we apply $\sigma_j$ and taking the product over all $j$'s we obtain another multiplicative relation of the form \begin{equation} \label{product} \left(\frac{\alpha_1}{\sqrt q}\right)^{S_{i,1}} \times \dots\times\left(\frac{\alpha_g}{\sqrt q}\right)^{S_{i,g}} = 1\end{equation}
where $S_{i,i} = \sum \pm k_j \equiv 1 \pmod2$. In particular $S_{i,i} \neq 0$.
This being true for each $i \in \lbrace 1,\dots, g\rbrace$, by taking a suitable product of large powers of expressions of the form \ref{product}, we deduce that there exists $S_1, \dots, S_g \in \Z \setminus \{0\}$ such that \begin{equation} \label{product2}\left(\frac{\alpha_1}{\sqrt q}\right)^{S_1} \times \dots\times\left(\frac{\alpha_g}{\sqrt q}\right)^{S_g} = 1.\end{equation}

Let $i \in \{1, \dots, g\}$. If $(i \, -i) \in \mathrm{Gal}_{\Q}(P_f)$, then we apply it to the relation \ref{product2} and taking a quotient we get $\left(\frac{\alpha_i}{\sqrt q}\right)^{2S_i} = 1$. This is a contradiction because $S_i \neq 0$ and $\frac{\alpha_i}{\sqrt q}$ is not a root of unity since $\chi_f$ is non-degenerate.

Now, let $i \neq j \in \{1, \dots, g\}$. If $(i \, j\, -i \, -j) \in \mathrm{Gal}_{\Q}(P_f)$ we get $\left(\frac{\alpha_i}{\sqrt{q}}\right)^{S_i+S_j} \left(\frac{\alpha_j}{\sqrt{q}}\right)^{S_j-S_i}= 1$, and similarly by applying its inverse $(i \, -j \, -i \, j) = (i \, j\, -i \, -j)^3$, we get $\left(\frac{\alpha_i}{\sqrt{q}}\right)^{S_i-S_j} \left(\frac{\alpha_j}{\sqrt{q}}\right)^{S_j+S_i}= 1$. Combining the two relations, we obtain $\left(\frac{\alpha_j}{\sqrt{q}}\right)^{(S_j+S_i)^2 + (S_j-S_i)^2 }= 1$.
But at least one among $S_i+S_j$ and $S_i - S_j$ is non-zero, since the $S_i$'s are non-zero, and as before, this shows that we cannot have $(i \, j \, -i \, -j) \in \mathrm{Gal}_{\Q}(P_f)$.
\end{proof}

\begin{lemma} \label{Lemma h reducible}
Let $P \in \mathbb Q[T]$ be a $q$-symplectic polynomial of degree $2g$ with roots $\alpha_1, \overline{\alpha_1}, \dots, \alpha_g, \overline{\alpha_g}$.
If $\Gal_{\Q}(P)$ does not act transitively on the pairs $\{\alpha_1, \overline{\alpha_1}\}, \dots, \{\alpha_g, \overline{\alpha_g}\}$ then $h_P$ defined by $P(T) = T^g h_P(T+qT^{-1})$ is reducible.
\end{lemma}

\begin{proof}
Notice the roots of $h_P$ are the $\alpha_i+ \overline{\alpha_i}$. Every element of $\Gal_{\Q}(h_P)$ are restrictions of elements of $\Gal_{\Q}(P)$ to the splitting field of $h_P$. Now if $h_P$ is irreducible over $\mathbb Q$, then $\Gal(h_P)$ acts transitively on the set $\{\alpha_i+ \overline{\alpha_i} \mid i=1, \dots, g\}$. Thus, if $i \neq j \in \{1, \dots, g\}$, there exists $\sigma \in \Gal(P)$ such that $\sigma(\alpha_i + \overline{\alpha_i}) = \alpha_j + \overline{\alpha_j}$. But $\sigma(\alpha_i) = \alpha_k$ for some $k \in \{1, \dots, 2g\}$ so we have $\sqrt q \cos(\theta_k) = \sqrt q \cos(\theta_j)$ which implies $\theta_k = \pm \theta_j$, which means $\sigma(\alpha_i) = \alpha_j$ or $\sigma(\alpha_i) = \overline{\alpha_j}$, and $\Gal_{\Q}(P)$ acts transitively on the set of pairs $\{\alpha_i, \overline{\alpha_i}\}$.
\end{proof}

We can finally prove the last part of our main theorem.

\begin{proof}[Proof of Theorem~\ref{thm:main}.\ref{item wrong direction th main}]
The proof follows by using the necessary conditions of Proposition \ref{Lemma necessary wrong direction}. We obtain a bound for the second condition by the same argument as in Lemma \ref{Lemma vanish 0}. For the first condition of Proposition \ref{Lemma necessary wrong direction}, we use Lemma~\ref{Lemma Galois condition WD} and bound the conditions $1.$ and $2.$ by Lemma \ref{Lemma bound quotient}. The third item is bounded by using Lemma \ref{Lemma h reducible} and Lemma \ref{Lemma bound real Weil}. Finally, the bound obtained with the condition on the Galois group, which is the largest contribution, is dealt with by Lemma \ref{Lemma bound wrong direction condition}.
\end{proof}

\subsection{Examples}

In the hope of finding examples of reversed bias in the sense of Definition \ref{wrongdirection} we estimated 
$$\Delta_f(n) = \left(m_{0}(\chi_f) +\tfrac{1}{2}\right)  +\left(m_{\pi}(\chi_f)+\tfrac{1}{2}\right) (-1)^{n} +\sum_{\theta_j\neq 0, \pi} m_{\theta_j}(\chi_f)  e^{in\theta_{j}(\chi_f)},$$ where $\chi_f$ is the primitive quadratic character modulo $f \in \mathcal{H}_n(\F_q)$ for small genera $g = \left \lfloor \frac{n-1}{2} \right \rfloor$ and small finite fields $\mathbb{F}_q$. In particular, for fixed $f(x) \in \mathbb{F}_q[x]$ we computed $\Delta_f(n)$ for many values of $n$, e.g. all $0 \leq n \leq 1000$. We found no clear candidate curves which exhibited a "strong" reversed bias amongst $\mathcal{C}_f/\mathbb{F}_q$ with $q$ a prime less than $11$ and $\deg f(x) \leq 6$ as well as among those curves with $\deg f(x) \leq 8$ and $q=3$.

\begin{remark}
    We can still provide an infinite family of examples exhibiting a reversed bias. Indeed, when $q$ is a square the polynomial $(1-u\sqrt{q}+u^2q)^2$ is the $L$-function of a hyperelliptic curve of genus $2$ according to \cite{HNR}. For such a curve $\mathcal{C}_f$, we have
$\Delta_f(n) = \frac{1}{2} + \frac{(-1)^n}{2} + 4\cos(\tfrac{2\pi}{3}n)$ which is $6$-periodic and takes $2$ positive values and $4$ negative values; explicitly, it takes the values $5,-2,-1,4,-1,-2$.
\end{remark}

Cha's example (\cite{Cha2008}*{Example 5.3}) corresponds to a reversed bias, however Cha is counting polynomials with degree less than $n$ instead of polynomials of degree equal to $n$ (see Remark \ref{Rk: Cha's counting}). We verified that this example does not meet our criterion of being a reversed bias with our way of counting  polynomials, but it exhibits a lower order bias because $\Delta_f(n)$ is $10$-periodic, takes $3$ positive values (at $n\in \lbrace0, 1, 9\rbrace$), $2$ negative values (at $n\in \lbrace3,7$) and is zero otherwise.

\section{A few counting lemmas}
\label{sec : counting}

Using the large sieve statement Proposition \ref{Prop Frobenius large sieve}, we will now prove important intermediate counting lemmas that are used to establish our upper bounds for exceptional biases. Recall that $q$ is a power of the prime $p$, and for any $n \geq 2$, $g = \left \lfloor \frac{n-1}{2} \right \rfloor$ is the genus of the curve $\mathcal{C}_f$ for any $f \in \mathcal{H}_n(\F_q)$ the set of monic squarefree polynomials in $\F_q[x]$ of degree $n$.

\begin{lemma} \label{Lemma vanish 0}
We have $$\frac{1}{\lvert \mathcal{H}_n(\mathbb{F}_q)\rvert }\#\left\{f \in \mathcal{H}_n(\mathbb{F}_q) \mid m_0(\chi_f) > m_{\pi}(\chi_f) \right\} = \begin{cases} 0 & \text{ if } q \text{ is a square}\\
O_{p,g} \big(q^{-\frac1{A}}\log q\big) & \text{ otherwise},
\end{cases}$$
where $A = 2g^2 +g +2$.
\end{lemma}

\begin{proof}

We first remark that the set $\left\{f \in \mathcal{H}_n(\mathbb{F}_q) \mid m_0(\chi_f) > m_{\pi}(\chi_f) \right\}$ is empty when $q$ is not a square, because in that case, $\sqrt q$ and $- \sqrt q$ are conjugate algebraic numbers, so they must have the same multiplicity as roots of a polynomial with integer coefficients such as $P_f$. 
We will prove our bound by showing that when $q$ is a square, we have $$\frac{1}{\lvert \mathcal{H}_n(\mathbb{F}_q)\rvert }\#\left\{f \in \mathcal{H}_n(\mathbb{F}_q) \mid m_0(\chi_f) > m_{\pi}(\chi_f)  \right\}
\leq \frac{1}{\lvert \mathcal{H}_n(\mathbb{F}_q)\rvert }\#\left\{f \in \mathcal{H}_n(\mathbb{F}_q) \mid m_0(\chi_f) \geq 1  \right\} \ll_{p,g} q^{-\frac1{A}}\log q.$$

For every $\ell \in \Lambda$ (recall that $\Lambda$ is simply the set of primes different from $2$ and $p$), we introduce the set $\Omega_{5, \ell} \subset \CSp_{2g}(\mathbb{F}_{\ell})$ of $q$-symplectic matrices for which $\sqrt q$ is not an eigenvalue. 
From \cite{Kowalski book}*{Lemma~B.5} (due to Chavdarov) we have
\begin{align*}
 \frac{\lvert \Omega_{5, \ell} \rvert}{\lvert \Sp_{2g}(\F_\ell)\rvert} \geq 1 - \frac{1}{\ell^g}\Big(\tfrac{\ell}{\ell - 1}\Big)^{2g^2 + g + 1}\#\lbrace  P \in \mathbb{F}_{\ell}[T], \text{ $q$-symplectic, } \deg{P} = 2g, P(\sqrt q) = 0  \rbrace.
\end{align*}

Since the set of symplectic $q$-polynomials of degree $2g$ in $\F_{\ell}[T]$ has dimension $g$, and that the condition of vanishing at one point is a linear equation of the coefficients, we have $$\# \lbrace  P \in \mathbb{F}_{\ell}[T], \text{ $q$-symplectic, } \deg{P} = 2g, P(\sqrt q) = 0  \rbrace = \ell^{g-1}.$$
We deduce that there exist a constant $C_g$ depending on $g$ such that
\begin{align*}
 \frac{\lvert \Omega_{5, \ell} \rvert}{\lvert \Sp_{2g}(\F_\ell)\rvert} \geq 1 - \frac{C_g}{\ell}. 
\end{align*}

Therefore, for $A = 2g^2 +g +2$, we have
\begin{align*}
\sum_{\substack{\ell \leq q^{\frac1{2A}} -1 \\ \ell \in \Lambda}} \frac{|\Omega_{5, \ell}|}{|\Sp_{2g}(\F_{\ell})| - |\Omega_{5, \ell}|}
\geq
\sum_{\substack{\ell \leq q^{\frac1{2A}} -1 \\ \ell \in \Lambda}} \frac{1 - \frac{C_g}{\ell}}{ \frac{C_g}{\ell}} 
\gg_{g} \sum_{\substack{\ell \leq q^{\frac1{2A}} -1 \\ \ell \in \Lambda}}
\ell 
\gg_{g} q^{\frac{1}{A}}(\log q)^{-1}.
\end{align*}
The desired bound then follows from Proposition~\ref{Prop Frobenius large sieve} by summing only over primes in $\Lambda$.
\end{proof}

\begin{remark}
We could improve the bound above by not restricting to the sum over primes, but we decided not to pursue this here, as we expect the improvement will only be on the power of $\log q$.
\end{remark}

The following lemma will allow us to reduce our counting to the case of non-degenerate characters $\chi_f$ (as in Definition \ref{nondegenerate}) and simple roots of $P_f$.

\begin{lemma}\label{Lemma bound quotient}
We have $$\frac{1}{\lvert \mathcal{H}_n(\mathbb{F}_q)\rvert }\#\left\{f \in \mathcal{H}_n(\mathbb{F}_q) \mid \chi_f \text{ is degenerate or } P_f \text{ has a multiple root in $\mathbb C$} \right\} \ll_{p,g} q^{-\frac1A}\log q,$$
where $A = 2g^2 +g +2$.
\end{lemma}

\begin{proof}
Let $f$ satisfy the above condition, that is $\chi_f$ is degenerate or $P_f$ has a multiple root in $\mathbb C$. Then there exist $1 \leq i \neq j \leq 2g$ such that $\frac{\alpha_i}{\alpha_j}$ is a root of unity, we denote $d$ its order (one can take $d=1$ in the case of a multiple root $\alpha_i=\alpha_j$). We first remark that $\alpha_i$ and $\alpha_j$ are algebraic integers of degree at most $2g$, so clearly $\frac{\alpha_i}{\alpha_j}$ is an algebraic number of degree at most $4g^2$, and so $\varphi(d) \leq 4g^2$.

Since $\alpha_i^d = \alpha_j^d$, it means that the polynomial $P_{f, (d)} = \prod_{i=1}^{2g} (X - \alpha_i^d)$ has a multiple root. This implies that its discriminant is $0$. Now, $\mathrm{disc}(P_{f, (d)})$ is a polynomial with integer coefficients in the coefficients of $P_{f, (d)}$ since it is the resultant of $P_{f, (d)}$ and its derivative. Moreover, those coefficients are symmetric polynomials in the $\alpha_k^d$'s, and in particular in the $\alpha_k$'s. By the fundamental theorem of symmetric polynomials, this is a polynomial expression in the elementary symmetric polynomials in the $\alpha_k$'s, which are precisely the coefficients of $P_f$.

We have shown that $P_f$ satisfies a certain integral polynomial equation, \textit{i.e.} there exists a polynomial $Q_{g, d} \in \Z[X_1, \dots, X_{2g}]$ such that, if $a_0, \dots, a_{2g-1}$ are the coefficients of $P_f$, then one has $Q_{g, d}(a_0, \dots, a_{2g-1}) = 0$. Since there are at most finitely many $d$ such that $\varphi(d) \leq 4g^2$, we get a universal relation $$Q_g = \prod_{d, \varphi(d) \leq 4g^2} Q_{g, d} \in \Z[X_1, \dots, X_{2g}]$$ such that if $\chi_f$ is degenerate or $P_f$ has a multiple root, then $Q_g(a_0, \dots, a_{2g-1}) = 0$.

Moreover, when $q$ is large enough, we know that $Q_g$ is non-zero since by Kowalski's result (Theorem~\ref{thm:Kowalski}), there exists a polynomial $h \in \F_q[x]$ monic of degree $n$ such that $P_h(T) =  T^{2g} + \dots + b_{1}T + b_0$ satisfies LI, and in particular, none of its quotients of roots is a root of unity, and for that polynomial, one has $Q(b_0, \dots, b_{2g-1}) \neq 0$.

So the equation $Q_g=0$ defines a hypersurface in  the set of $q$-symplectic polynomials of fixed degree, and we have
\begin{align*}
    \#\lbrace P \in \mathbb{F}_{\ell}[T] \mid &\text{ $q$-symplectic, } \deg{P} = 2g, Q(P) =0\rbrace \ll_g \ell^{g-1}.
\end{align*}
See also \cite{Kowalski book}*{Theorem~B.6}.
The end of the proof is completely similar to the end of the proof of Lemma~\ref{Lemma vanish 0}.
\end{proof}

In the next lemma, $h_{P_f}$ denotes the ``real Weil polynomial'' attached to $\mathcal C_f$, defined by the relation $$P_f(T) = T^g h_{P_f}(T + qT^{-1}).$$

\begin{lemma} \label{Lemma bound real Weil}
We have $$\frac{1}{\lvert \mathcal{H}_n(\mathbb{F}_q)\rvert }\#\left\{f \in \mathcal{H}_n(\mathbb{F}_q) \mid h_{P_f} \text{ is reducible}\right\} \ll_{p,g} q^{-\frac{1}{2A}}(\log q)^{1- \frac{1}{g-1}},$$
where $A = 2g^2 +g +2$.
\end{lemma}

\begin{proof} We use Proposition~\ref{Prop Frobenius large sieve} with the set
$$\Omega_{6,\ell} = \lbrace M \in \CSp_{2g}(\mathbb{F}_{\ell}) \mid \exists h \in \F_{\ell}[T], h \text{ is monic irreducible and } \chi_M (T) = T^g h(T + q T^{-1})\rbrace.$$
Since if a monic polynomial is reducible, none of its reduction modulo a prime can be irreducible, we have

\begin{align*}
\#\left\{f \in \mathcal{H}_n(\mathbb{F}_q) \mid h_{P_f} \text{ is reducible}\right\}  &\leq \# \left\{ f \in \mathcal{H}_n(\F_q) \mid \rho_{\ell}(\frob_{f,q}) \notin \Omega_{6,\ell} \text{ for all } \ell < q^{\frac{1}{2A}}, \ell\in \Lambda \right\} \\ 
&\ll_{p,n} \lvert \mathcal{H}_n(\mathbb{F}_q)\rvert \Big( \sum_{\substack{\psi(m)\leq q^{\frac1{2A}} \\ m \in \mathcal{L}}} \prod_{\ell \mid m}\frac{\delta_{6,\ell}}{1 - \delta_{6,\ell}}\Big)^{-1}.
\end{align*}
where $\delta_{6,\ell} = \frac{\lvert \Omega_{6,\ell} \rvert}{\lvert \Sp_{2g}(\F_\ell)\rvert}$.
There are $\frac1g \ell^g(1 + O_g(\tfrac1\ell))$ monic irreducible polynomials of degree $g$ with coefficients in $\mathbb{F}_\ell$. As $P \mapsto h_P$ is a bijection from the set of $q$-symplectic polynomials in $\F_{\ell}[T]$ of degree $2g$ to the set of monic polynomials of degree $g$ in $\F_{\ell}[T]$, we deduce from \cite{Kowalski book}*{Lemma B.5} (similarly to Lemma~\ref{Lemma vanish 0}) that
$$\delta_{6, \ell} \geq \frac1g + O_g(\tfrac{1}{\ell}) =: \delta_{\ell}.$$
We conclude using the estimation of the sum from a theorem of Lau and Wu \cite{Kowalski book}*{Theorem G.2} applied the same way as Kowalski in \cite{Kowalski book}*{(8.24)}:
 \begin{align*}
 \sum_{\substack{\psi(m)\leq q^{\frac{1}{2A}} \\ m \in \mathcal{L}}} \prod_{\ell \mid m}\frac{\delta_{6, \ell}}{1 - \delta_{6, \ell}} \geq
    \sum_{\substack{\psi(m)\leq q^{\frac{1}{2A}} \\ m \in \mathcal{L}}} \prod_{\ell \mid m}\frac{\delta_{\ell}}{1 - \delta_{\ell}} \gg
    q^{\frac{1}{2A}}(\log q)^{-1 + \frac{1/g}{1 - 1/g}},
 \end{align*}
from which we deduce the stated bound.
\end{proof}

The last counting lemma is about polynomials $f \in \mathcal{H}_n(\F_q)$ such that $\mathrm{Gal}_{\mathbb Q}(P_f)$ does not contain certain permutations. Recall from the discussion above Lemma \ref{Lemma Galois condition WD} that $\mathrm{Gal}_{\mathbf{Q}}(P_f)$ acts on $\{-g, \dots, -1, 1, \dots, g\}$.

\begin{lemma}\label{Lemma bound wrong direction condition}
We have
\begin{align*}
    \#\Big\{f \in \mathcal{H}_n(\mathbb{F}_q) \mid \forall i \neq j \in \{1, \dots, g\},  (i \, -i) \notin\mathrm{Gal}(P_f) &\text{ and } (i \, j \, -i \, -j) \not \in \mathrm{Gal}(P_f)  \Big\}  \\
    &\ll_{p,g} q^{-\frac{1}{2A}} (\log q)^{1-\frac{1}{\frac{24}7 g-1}},
\end{align*}
where $A=2g^2+g+2$.
\end{lemma}

\begin{proof}
First, we may assume that $P_f$ is separable, since the announced bound is worse than that of \ref{Lemma bound quotient}.

We are once again going to use the large sieve bound coming from Proposition~\ref{Prop Frobenius large sieve} but the set $\Lambda$ of prime numbers used in the large sieve has to be modified a bit here because of Lemma \ref{lemma counting polynomials}: we take $\Lambda$ to be the set of prime numbers different from $2$ and $p$ and larger than $4g^2$ (see Remark \ref{rk : change Lambda}). This only induces a further dependency on $g$ in the implied constants, but doesn't modify the final bound.

For every $\ell \in \Lambda$, we consider $\Omega_{7,\ell}$ be the set of $q$-symplectic matrices $M \in \CSp_{2g}(\mathbb{F}_{\ell})$ such that the characteristic polynomial $\chi_M$ admits a factorization either as a quadratic irreducible polynomial multiplied by distinct irreducible polynomials of odd degree, or as a quartic irreducible polynomial multiplied by distinct irreducible polynomials of odd degree. 
Indeed, if $P_f$ is separable but the Galois group $\mathrm{Gal}(P_f)$ does not contain a transposition nor a $4$-cycle (when seen as a subgroup of $\mathfrak{S}_{2g}$), then $\rho_{\ell}(\frob_{f,q}) \notin \Omega_{7,\ell}$ 
for any $\ell$ (see \cite{Jacobson}*{Theorem 4.37}).

Therefore, we need to count the symplectic polynomials with such factorizations to be able to conclude as above. For $\ell\in \Lambda$, we let $\delta_{7, \ell} = \frac{\lvert \Omega_{7,\ell} \rvert}{\lvert \Sp_{2g}(\F_\ell)\rvert}$.

In the case $g$ is even, we use Lemma~\ref{lemma counting polynomials} with 
 ($k=1$, $n_{\frac{g-2}{2}} = 1$) and with ($k=2$, $n_{\frac{g-4}{2}} = 1$, $n_{0} = 1$) to get
$$\delta_{7, \ell} \geq \frac1{4(g-1)} + \frac1{16(g-3)} + O(\ell^{-1/2}) \geq \frac5{16g} + O(\ell^{-\frac12}).$$

In the case $g$ is odd we use Lemma~\ref{lemma counting polynomials} with
($k=1$, $n_{\frac{g-3}{2}} = 1$, $n_0=1$) with ($k=1$, $n_{\frac{g-5}{2}} = 1$, $n_{1} = 1$), and with ($k=2$, $n_{\frac{g-3}{2}} = 1$) to get 
$$\delta_{7, \ell} \geq \frac1{8(g-2)} +  \frac1{24(g-4)} + \frac1{8(g-2)} + O(\ell^{-\frac12}) \geq \frac7{24g} + O(\ell^{-\frac12}).$$

In both cases we have $\delta_{7, \ell} \geq \frac7{24g} + O(\ell^{-\frac12})$, so we obtain the announced bound in the same way as in the proof of Proposition~\ref{Lemma bound real Weil}.
\end{proof}

\begin{remark}
In the proof above of Lemma \ref{Lemma bound wrong direction condition} one could expand the application of Lemma \ref{lemma counting polynomials} to add more terms to the lower bound of $\delta_{7,\ell}$ to gain marginal  improvements. The additional condition in the lemma and its application above with $k=2$ delivers our improvement over Kowalski's bound \eqref{Kowalski}.
\end{remark}

\begin{lemma}\label{lemma counting polynomials}
 Let $ 0 \leq k\leq g$ be two integers, and let $\ell>4g^2$ be a prime number. Write $r = \lfloor \frac{g-k-1}{2}\rfloor$ and let $n_i$, $1\leq i \leq r$, 
be integers such that
$g = k + n_0  + 3n_1 + 5n_2 + \dots + (2r+1)n_{r}$ .
Let $\omega_{k,\ell}(\underline{n})$ be the set of $q$-symplectic squarefree polynomials $P \in \mathbb{F}_{\ell} [T ]$ which factor as
a product
$P = Q_{2k} R_0 \tilde{R_0} R_1 \tilde{R_1} \dots R_r \tilde{R_r}$ ,
where $Q_{2k}$ is an irreducible $q$-symplectic polynomial of degree $2k$, each $R_i$ is a product of $n_i$ distinct irreducible monic polynomials of degree $2i+1$, and $\tilde{R_i} = \frac{T^{(2i+1)n_i}}{R_i(0)}R_i\left(\frac{q}{T}\right)$ is the $q$-reciprocal of $R_i$.
Then, we have
$$\lvert \omega_{k,\ell}(\underline{n})\rvert \geq \left(\prod_{i=0}^{r} \frac{1}{2^{n_i}(2i+1)^{n_i} n_i!}\right)  \frac{1}{2k}\ell^{g} - O(\ell^{g-\frac12}).$$
\end{lemma}

\begin{proof}
First observe that for any $q$-symplectic polynomial $P \in \F_{\ell}[T]$, one has $P(0)=q^{\deg P/2} \neq 0$, in particular, for all $R \mid P$, one has $R(0) \neq 0$.
We appeal to \cite{Kowalski2006}*{Lemma 7.3 (ii)}, which gives that the count of irreducible symplectic polynomials of degree $2k$ is larger than $\frac1{2k}\ell^{k} - O(\ell^{k-1})$ (see also \cite{DDS}*{Lemma~3} which can be adapted to the case of $q$-symplectic polynomials). The irreducible factors of odd degree of a symplectic polynomial come in pairs $\{R(T), \tilde{R}(T) = \frac{T^{\deg R}}{R(0)}R\left(\frac{q}{T}\right)\}$, uniquely determined by either of its elements. So it suffices to count polynomials of degree $g-k$ that are products of distinct odd degree irreducible polynomials. By \cite{Kowalski book}*{Lemma~B.1} there are $\prod_{i=0}^{r} \frac{1}{(2i+1)^{n_i} n_i!} \ell^{g-k} - O(\ell^{g-k-\frac12})$ polynomials with given factorization $R_0R_1\dots R_r$ (as in the statement of the lemma). 

For each polynomial with factorization type $R_0R_1\dots R_r$, for each factor $R_i$, $0 \leq i \leq r$, we have made a choice of which element of the pair $\{R(T), \tilde R(T) \}$ to include. There are $2^{n_i}$ such choices for each $i$.

We just need to remove from the final count the monic $q$-symplectic polynomials that have multiple roots, as counted in the proof of Lemma~\ref{Lemma bound quotient}, there are at most $O(\ell^{g-1})$ such polynomials, so this does not change the main term.
\end{proof}

\end{document}